%% file: MarkovAdditiveMain.tex
\documentclass[pdftex,11pt,svgnames]{article}
\usepackage{adjustbox}
\usepackage[utf8]{inputenc}            
\usepackage{amsmath}
\usepackage{amssymb}  
\usepackage{amsbsy} 
\usepackage{listings}
\usepackage{graphics}
\usepackage{bm}
\usepackage{prodint}

\usepackage[multiple]{footmisc}

\usepackage{enumerate}
\usepackage[shortlabels]{enumitem}

\IfFileExists{dsfont.sty}{
	\usepackage{dsfont}
	\usepackage[usenames,dvipsnames]{color} 
	
}{
	
}

\usepackage[numbers]{natbib}

\usepackage{pdfpages}
\usepackage{hyperref} 
\hypersetup{
    unicode=false,          % non-Latin characters in Acrobat’s bookmarks
    pdftoolbar=true,        % show Acrobat’s toolbar?
    pdfmenubar=true,        % show Acrobat’s menu?
    pdffitwindow=false,     % window fit to page when opened
    pdfstartview={FitH},    % fits the width of the page to the window
    pdftitle={Aggregate Markov models in life insurance: properties and valuation},    % title
    pdfauthor={Jamaal Ahmad, Mogens Bladt, and Christian Furrer},     % author
    pdfsubject={Mathematics of life insurance},   % subject of the document
    pdfcreator={ },   % creator of the document
    pdfproducer={ }, % producer of the document
    pdfkeywords={ } { } { }, % list of keywords
    pdfnewwindow=true,      % links in new window
    colorlinks=true,       % false: boxed links; true: colored links
    linkcolor=Black, %WildStrawberry,          % color of internal links
    citecolor=Black,      % color of links to bibliography
    filecolor=Black,      % color of file links
    urlcolor=Black 		  %WildStrawberry  % color of external links
}
\usepackage[font=small,format=plain,labelfont=bf,up,textfont=it,up]{caption}

\usepackage[a4paper]{geometry}  % Geometri-pakke: Styrer bl.a. maginer    %
\usepackage[english]{babel}

\usepackage[autolanguage,np]{numprint}

\usepackage{doi}
\usepackage{float}

\usepackage{fancyhdr}
\usepackage{amsthm}
\theoremstyle{plain}
\newtheorem{Theorem}{Theorem}
\newtheorem{theorem}[Theorem]{Theorem}

\newtheorem{lemma}[Theorem]{Lemma}

\newtheorem{proposition}[Theorem]{Proposition}
 
\newtheorem{corollary}[Theorem]{Corollary} 

\newtheorem{definition}[Theorem]{Definition}

\newtheorem{conjecture}[Theorem]{Conjecture}

\theoremstyle{definition}

\newtheorem{example}[Theorem]{Example}

\newcommand\xqed[1]{%
  \leavevmode\unskip\penalty9999 \hbox{}\nobreak\hfill
  \quad\hbox{#1}}

\newcommand\demoo{\xqed{$\triangle$}}

\newcommand{\demoeq}{\tag*{$\circ$}}
\newcommand{\demooeq}{\tag*{$\triangle$}}

\theoremstyle{remark}

\newtheorem{remark}[Theorem]{Remark}

\numberwithin{Theorem}{section}

\numberwithin{equation}{section}

\usepackage{algorithm}
\usepackage{algorithmicx}
\usepackage{algpseudocode}
\newcommand{\md}{{\, \mathrm{d} }}

\newcommand{\tuborg}[1]{\{ #1 \}}

\newcommand{\Prob}{\mathbb{P}}

\DeclareMathAlphabet{\mathpzc}{OMS}{pzc}{m}{it}
\newcommand{\J}{\mathpzc{J}}

\newcommand{\N}{\mathbb{N}}

\newcommand{\vect}[1]{\pmb{#1}}
\newcommand{\mat}[1]{\boldsymbol{\bm #1}}

\makeatletter

\newcommand{\Rmnum}[1]{\expandafter\@slowromancap\romannumeral #1@}
\makeatother

\usepackage{wasysym} % Nu får du cirkler i slutningen af dine exempler og bemærkninger...

\makeatletter
\newcommand*\bigcdot{\mathpalette\bigcdot@{.5}}
\newcommand*\bigcdot@[2]{\mathbin{\vcenter{\hbox{\scalebox{#2}{$\m@th#1\bullet$}}}}}
\makeatother

\usepackage{tikz}
\usetikzlibrary{arrows,positioning,calc} 

\tikzset{
    >=stealth',
    punkt/.style={
           rectangle,
           rounded corners,
           draw=black, thick,
           text width=7em,
           minimum height=2em,
           text centered},
    punktl/.style={
           re
           tangle,
           rounded corners,
           draw=black, thick,
           
           text width=7em,
           minimum height=2em,
           text centered},
    pil/.style={
           ->,
           shorten <=4pt,
           shorten >=4pt,},
    pildotted/.style={
           ->,
           shorten <=4pt,
           shorten >=4pt,
  dotted,
  }
}

\usepackage{subfig}
\usepackage{todonotes}

\usepackage{authblk}

\allowdisplaybreaks

\makeatletter
\let\@fnsymbol\@arabic
\makeatother

\usepackage[scr=boondoxo]{mathalfa}
\newcommand\iay{\mathscr{i}}
\newcommand\jay{\mathscr{j}}
\newcommand\kay{\mathscr{k}}

\usepackage{etoolbox}
\usepackage[title]{appendix}
\patchcmd{\appendices}{\quad}{: }{}{}

\title{Aggregate Markov models in life insurance: properties and valuation}
\author[1,$\star$]{Jamaal Ahmad}
\author[1]{Mogens Bladt}
\author[1]{Christian Furrer}
\affil[1]{\footnotesize Department of Mathematical Sciences, University of Copenhagen, Universitetsparken 5, DK-2100 Copenhagen \O, Denmark.}
\affil[$\star$]{\footnotesize Corresponding author. E-mail: \href{mailto:jamaal@math.ku.dk}{jamaal@math.ku.dk}.}

\date{ }

\begin{document}
\maketitle{}

%\addtocounter{footnote}{4} % Fire fodnoter er brugt i author...

\input{MarkovAdditiveAbstract} 
\input{MarkovAdditiveContents}  
\input{MarkovAdditiveAppendix} 

%\bibliography{refsSmall}{}
%\bibliographystyle{plainnat}

%\bibliography{refsSmall}{}
%\bibliographystyle{plain}

\bibliographystyle{plain}
\bibliography{MarkovAdditiveMain.bbl}{}

\end{document}

%% file: MarkovAdditiveAbstract.tex
\begin{center}
{\sc Abstract}
\end{center}

{\small In multi-state life insurance, an adequate balance between analytic tractability, computational efficiency, and statistical flexibility is of great importance. This might explain the popularity of Markov chain modelling, where matrix analytic methods allow for a comprehensive treatment.  Unfortunately, Markov chain modelling is unable to capture duration effects, so this paper presents aggregate Markov models as an alternative. Aggregate Markov models retain most of the analytical tractability of Markov chains, yet are non-Markovian and thus more flexible. Based on an explicit characterization of the fundamental martingales, matrix representations of the expected accumulated cash flows and corresponding prospective reserves are derived for duration-dependent payments with and without incidental policyholder behaviour. Throughout, special attention is given to a semi-Markovian case. Finally, the methods and results are illustrated in a numerical example.

\vspace{5mm}

\vspace{5mm}



}

%2020 MSC
%91G05 Actuarial mathematics
%60J27 Continuous-time Markov processes on discrete state spaces
%60J28 Applications of continuous-time Markov processes on discrete state spaces
%91G60 Numerical methods (including Monte Carlo methods)
 	 	 	
%JEL Classification
%G22: Insurance • Insurance Companies • Actuarial Studies 
%C63: Computational Techniques • Simulation Modeling 

%% file: MarkovAdditiveContents.tex
%!TEX root = MarkovAdditiveMain.tex

\section{Introduction}\label{sec:intro}

In this paper, we propose a new class of multi-state models, the so-called aggregate Markov models, and study the valuation of life insurance liabilities for this class of models. The results established range from a characterization of the fundamental martingales to genuine computational schemes for the quantities of interest, including prospective reserves. In contrast, the companion paper~\cite{AhmadBladt} deals with the statistical aspects. 

The classic approach to multi-state modelling consists of Markov chain modelling, where the process governing the state of the insured $Z$ is taken to be a time-inhomogeneous Markov chain and the payments between the insured and the insurer are required to consist of deterministic sojourn and transition payments. This approach dates back to at least~\cite{hoem69}, but has recently been given new life via matrix analytic methods, see~\cite{Bladt2020,Ahmad2022}. Although Markov chain modelling is attractive due to its inherent simplicity, it suffers from a number of defects, not at least its inability to properly capture duration effects. In recent years, semi-Markov modelling has therefore gained considerable attention. In semi-Markov modelling, see~\cite{Hoem1972,helwich,christiansen2012,BuchardtMollerSchmidt}, only the joint process $(Z,U)$, where $U$ denotes the time spent in the current state, is assumed to be Markovian, and the sojourn and transition payments are allowed to depend on $U$. Unfortunately, semi-Markov modelling entails less analytical tractability and increased computational load. 

In an aggregate Markov model, each observable state is assumed to consist of multiple unobservable sub-states, and only the full model consisting of all sub-states is assumed to be Markovian. Different to the classic Markov chains, which require the Markov property already for the observable states, aggregate Markov models are non-Markovian and thus flexible, yet they retain most of the analytical tractability of Markov chains. In particular, matrix analytic methods related to inhomogeneous phase-type distributions, confer with~\cite{Albrecher-Bladt-2019}, are applicable and lead to a unifying and transparent treatment. 

Phase-type distributions have a long history of extensive use in applied probability. They have been employed in areas such as queueing theory~\cite{asmussen2003applied,neuts1981,neuts2021structured,latouche1999introduction}, actuarial science~\cite{ruin,bladt2017matrix}, and telecommunications~\cite{asmussen2003applied,latouche1999introduction}, where the phase-type assumption leads to exact and in many cases explicit formulas for properties such as waiting time distributions, queue length, ruin probabilities, and buffer overflows. 

Both homogeneous as well as inhomogeneous phase-type distributions are dense in the class of distributions on the positive reals, confer with~\cite{bladt2017matrix}, and therefore able to approximate any non-negative distribution arbitrarily close -- in the sense of weak convergence as the number of phases increases to infinity. Hence the class of phase-type distributions has been considered as striking a balance between tractability and generality. Inhomogeneous phase-type distributions may be used instead of homogeneous ones, and this might be particularly relevant if the tail behaviour is known to be different from exponential, see~\cite{Albrecher-Bladt-2019}.

In the area of queuing theory, so-called quasi-birth-and-death (QBD) processes have been extensively studied, confer with~\cite{latouche1999introduction} {and references therein}, as a model for the number of customers in a queue. They constitute the time-homogeneous analogue to the aggregate Markov models considered here.

The first main contribution of the paper is an explicit characterization of the martingales for the associated counting processes $N$, which reveals that aggregate Markov models may be highly non-Markovian. For many practical purposes, less might suffice. To this end, we provide a sort of \textit{reset property} under which the aggregate Markov model is actually semi-Markovian. The second main contribution of the paper are matrix representations for the expected accumulated cash flows, and hereby the prospective reserves, for duration-dependent payments with and without incidental policyholder behaviour. Special attention is given to the case where the payments are duration independent; here our results indicate that aggregate Markov modelling may hold a competitive advantage over semi-Markov modelling.

The remainder of the paper is structured as follows. Section~\ref{sec:prelims} provides some background, with Subsection~\ref{sec:inhom_phasetype} devoted to the basics of inhomogeneous phase-type distributions and Subsection~\ref{sec:multi-state} to the basics of multi-state modelling in life insurance. These subsections might be passed over by readers who are familiar with the subject matter. In Section~\ref{sec:setup}, we introduce the aggregate Markov models and state the aforementioned reset property. The main contributions take place in Section~\ref{sec:propZ} and Section~\ref{sec:reserves_cf}. The former is devoted to the distributional properties of $Z$, including the characterization of the fundamental martingales, while the latter deals with the valuation of life insurance liabilities and contains, in particular, matrix representations of the expected accumulated cash flows. To showcase the practical potential of aggregate Markov modelling, {Section~\ref{sec:num} contains a numerical example. Finally, Section~\ref{sec:conc} concludes.} Proofs may be found in Appendix~\ref{ap:A}. 

\section{Preliminaries}\label{sec:prelims}

Before introducing the setting of the paper, we provide some background. Subsection~\ref{sec:inhom_phasetype} contains a short review on imhomogeneous phase-type distributions, which play a critical role later. This review is followed by Subsection~\ref{sec:multi-state}, which collects some insights on multi-state modelling in life insurance, in particular in relation to Markov chain and semi-Markov modelling, hereby motivating our aggregate setup. The actual presentation of our setup is postponed to Section~\ref{sec:setup}.

In what follows and throughout the paper, we denote the product integral of a square matrix function $\mat{A}(x)$ as  
\begin{align*}
\mat{F}(t,s)=\Prodi_t^s \!\left(\mat{I}+\mat{A}(x)\md x\right)\!,
\end{align*}
where $\mat{I}$ is the identity matrix. Under suitable regularity conditions, it may equivalently be cast as the solution to Kolmogorov's forward and backward differential equations:
\begin{align*}
\frac{\partial}{\partial s}\mat{F}(t,s)&=\mat{F}(t,s)\mat{A}(s), &&\mat{F}(t,t) = \mat{I}, \\
\frac{\partial}{\partial t}\mat{F}(t,s)&=-\mat{A}(t)\mat{F}(t,s), &&\mat{F}(s,s) = \mat{I}. 
\end{align*}
{ In other words, the computation of a product integral is equivalent to solving a system of first order, in general inhomogeneous, ordinary differential equations. In particular, a standard Runge--Kutta scheme may be employed. Alternatively, the solution may be approximated by a product of matrix exponentials, using a piece-wise constant square matrix function. Further details may be found in~\cite[Section~2]{Bladt2020}.} For a survey on {the probabilistic and statistical aspects of} product integration, we refer to~\cite{GillJohansen} and, concerning applications to life insurance, also~\cite{milbrodtstracke1997,Bladt2020}. 

\subsection{Inhomogeneous phase-type distributions}\label{sec:inhom_phasetype}

In this subsection, we review the notion of inhomogeneous phase-type (IPH) distributions introduced in~\cite{Albrecher-Bladt-2019}. Consider a smooth and suitably regular time-inhomogeneous Markov jump process $X=\{X(t)\}_{t\geq 0}$ on the finite state space $\J = \{1,\ldots,J-1,J\}$, where the states $\{1,\ldots,J-1\}$ are transient while $J$ is absorbing. The transition intensity matrix function $\mat{M}(t) = \{\mu_{ij}(t)\}_{i,j\in \J}$ of $X$ is then on the form
\begin{align*}
\mat{M}(t) = 
\begin{pmatrix}
\mat{T}(t) & \mat{t}(t) \\
0 & 0
\end{pmatrix}\!,
\end{align*}
where $\mat{T}(t)$ is a sub-intensity matrix function consisting of transition rates between the transient states and $\mat{t}(t) = -\mat{T}(t)\vect{1}_{J}$ is a column vector of transition rates to the absorbing state, the so-called exit rate vector function. Further, assume that $\Prob(X(0) = J)=0$ and denote by $\mat{\pi} = (\pi_1,\ldots,\pi_{J-1})$ the remaining vector of initial probabilities $\pi_j = \Prob(X(0) = j)$. The time until absorption, given by 
\begin{align*}
\tau = \inf\{t\geq 0 \, : \, X(t) = J\},
\end{align*}
is then said to be an inhomogeneous phase-type distribution with representation $(\vect{\pi},\vect{T})$, and we write $\tau\sim \mathrm{IPH}(\vect{\pi},\vect{T})$.  

The transition probability matrix function $\mat{P}(t,s) = \{p_{ij}(t,s)\}_{i,j\in \J}$ with elements
\begin{align*}
p_{ij}(t,s) = \Prob\!\left(\!\left. X(s) = j\, \right|  X(t) = i\right)
\end{align*}
is given as the product integral of the transition intensity matrix function:
\begin{align*}
\mat{P}(t,s) = \Prodi_t^s \!\left(\mat{I}+\mat{M}(x)\md x\right)\!. 
\end{align*}
The probability density function $f(t)$ and distribution function $F(t)$ of $\tau$ may then be obtained through product integrals of the sub-intensity matrix function ${\mat{T}}(t)$:
\begin{align}
f(x) &= \vect{\pi}\Prodi_0^x \!\left(\mat{I}+\mat{T}(u)\md u\right)\!\vect{t}(x), \\ \label{eq:dist_prodint}
F(x) &= 1-\vect{\pi}\Prodi_0^x \!\left(\mat{I}+\mat{T}(u)\md u\right)\!\vect{1}_{J}.
\end{align}
From these, one finds the following conditional distribution: 
\begin{align}\label{eq:cond_IPH}
\Prob\!\left(\left. \tau > s+t\, \right| \, \tau > s\right) = \frac{\vect{\pi}\displaystyle\Prodi_0^s \!\left(\mat{I}+\mat{T}(x)\md x\right)\! }{\vect{\pi}\displaystyle\Prodi_0^s \!\left(\mat{I}+\mat{T}(x)\md x\right)\!\vect{1}_{J}}\Prodi_s^{s+t} \!\left(\mat{I}+\mat{T}(x)\md x\right)\!\vect{1}_{J},
\end{align}
which entails that
\begin{align*}
\tau - s\, | \, \tau > s \sim \mathrm{IPH}\!\left(\vect{\alpha}(s),\vect{T}(s+\bigcdot)\right)\!, 
\end{align*}
where  $\vect{\alpha}(s)$ is given by
\begin{align*}
\vect{\alpha}(s) = \frac{\vect{\pi}\displaystyle\Prodi_0^s \!\left(\mat{I}+\mat{T}(x)\md x\right)\! }{\vect{\pi}\displaystyle\Prodi_0^s \!\left(\mat{I}+\mat{T}(x)\md x\right)\!\vect{1}_{J}}. 
\end{align*}
In other words, for IPH distributions, the overshoot is again IPH-distributed.
\begin{example}\label{ex:singe_phase}
In the case of a single phase, that is $J=2$, we have 
\begin{align*}
\Prodi_t^s \!\left(\mat{I}+\mat{T}(x)\md x\right) = e^{-\int_t^s\mu_{12}(x)\md x},
\end{align*}
giving the density and distribution functions
\begin{align*}
f(x) &= e^{-\int_0^x \mu_{12}(v)\md v}\mu_{12}(x), \\
F(x) &= 1-e^{-\int_0^x \mu_{12}(v)\md v},
\end{align*}
while the conditional distribution~\eqref{eq:cond_IPH} takes the form
\begin{align*}
\Prob\!\left(\left. \tau > s+t\, \right| \, \tau > s\right) = \frac{e^{-\int_0^s \mu_{12}(v)\md v} }{e^{-\int_0^s \mu_{12}(v)\md v}}e^{-\int_s^{s+t} \mu_{12}(v)\md v} = e^{-\int_s^{s+t} \mu_{12}(v)\md v}.\demoeq
\end{align*}
\end{example}

\subsection{Multi-state modelling}\label{sec:multi-state}

Insurance contracts may be modelled as a stream of payments $B=\{B(t)\}_{t\geq0}$, benefits less premiums, between the insured and the insurer. In life insurance, including health and disability insurance and pensions, the payments depend on the state of the insured, leading to so-called multi-state modelling. In general, the state of the insured $Z=\{Z(t)\}_{t\geq0}$ is a non-explosive jump process on a typically finite state space $\mathcal{J} = \{1,2,\ldots,J\}$, $J\in\mathbb{N}$, while the payments are typically finite variation processes adapted to the information generated by $Z$.
 
\subsubsection*{Markov chain models}

The most classic approach to multi-state modelling is (smooth) Markov chain models, where $Z$ is taken to be a time-inhomogeneous Markov jump process (Markov chain) with suitably regular transition rates $\nu_{jk}(t)$, so that
\begin{align*}
\nu_{jk}(t) = \lim_{h \downarrow 0} \frac{\mathbb{P}(Z(t+h) = k \, | \, Z(t) = j)}{h}.
\end{align*}
Using the standard convention $\nu_{jj}(t) = - \sum_{k \in \mathcal{J}\atop k \neq j} \nu_{jk}(t)$, the square matrix function with indices $\nu_{jk}(t)$ is then the transition intensity matrix function of $Z$. In addition to the Markov assumption, the payments are assumed to take the form
\begin{align*}
\mathrm{d}B(t) = \sum_{j \in \mathcal{J}} \bigg( 1_{(Z(t) = j)} b_j(t)\, \mathrm{d}t + \sum_{k \in \mathcal{J}\atop k \neq j} b_{jk}(t) \, \mathrm{d}N_{jk}(t)\bigg), \quad B(0) \in \mathbb{R},
\end{align*}
for suitably regular deterministic sojourn payment rates $b_j(t)$ and transition payments $b_{jk}(t)$ depending only on time. Here $N$ is the multivariate counting process associated to $Z$ with components $N_{jk} = \{N_{jk}(t)\}_{t\geq0}$ given by
\begin{align*}
N_{jk}(t)=\# \{s\in (0,t] : Z(s-)=j, Z(s) = k \}.
\end{align*}
Markov chain modelling dates back to at least~\cite{hoem69} and was popularized in~\cite{norberg1991}.

Regarding the valuation of life insurance liabilities, calculating the so-called expected accumulated cash flow $A(t,s)$ is key. For Markov chain models, the expected accumulated cash flow $A(t,s)$ is given by $A(t,s) = \sum_{i\in\mathcal{J}} 1_{(Z(t) = i)} A_i(t,s)$, where
\begin{align*}
A_i(t,s)
&=
\mathbb{E}[ B(s) - B(t) \, | \, Z(t) = i] \\
&=
\sum_{j \in \mathcal{J}} \int_t^s \mathbb{P}(Z(u) = j \, | \, Z(t) = i) \bigg( b_j(u) + \sum_{k \in \mathcal{J}\atop k \neq j} \nu_{jk}(u) b_{jk}(u)\bigg) \mathrm{d}u.
\end{align*}
The transition probabilities, considered as a square matrix function, are given as the product integral of the transition rates, also considered as a square matrix function. In other words, the transition probabilities may be found simply by solving Kolmogorov's forward differential equations.

It is, of course, possible, and of interest, to relax both the Markov assumption as well as the structure of the payments. In doing so, it is critical to strike an adequate balance between analytic tractability, computational efficiency, and statistical flexibility.

\subsubsection*{Semi-Markov models}

A more modern approach is semi-Markov modelling, see for instance~\cite{helwich,christiansen2012,BuchardtMollerSchmidt}. There are two major differences between (smooth) semi-Markov modelling and (smooth) Markov chain modelling. First, the jump process $Z$ describing the state of the insured is no longer required to be Markovian; rather, $(Z,U)$ is assumed Markovian, where $U=\{U(t)\}_{t\geq 0}$ is the duration since the last transition given by
\begin{align}\label{eq:U_def}
U(t) = \sup\!\big\{ s\in [0,t] :  Z(u) = Z(t)  \text{ for all }  u\in [t-s,t]\big\}.
\end{align}
Therefore, the model can no longer be described by transition rates that solely depend on time. Instead, the transition rates are now functions of both time and duration, written $\nu_{jk}(t,u)$.

Second, the payments take the more general form
\begin{align}\label{eq:payments_duration}
\mathrm{d}B(t) = \sum_{j \in \mathcal{J}} \bigg( 1_{(Z(t) = j)} b_j(t,U(t))\, \mathrm{d}t + \sum_{k \in \mathcal{J}\atop k \neq j} b_{jk}(t,U(t-)) \, \mathrm{d}N_{jk}(t)\bigg), \quad B(0) \in \mathbb{R},
\end{align}
for suitably regular deterministic sojourn payment rates $b_j(t,u)$ and transition payments $b_{jk}(t,u)$ depending on time and duration.

For semi-Markov models, the expected accumulated cash flow $A(t,s)$ depends on both the current state and current duration. To clarify, it may actually be decomposed according to $A(t,s) = \sum_{i\in\mathcal{J}} 1_{(Z(t) = i)} A_{i,U(t)}(t,s)$, where
\begin{align*}
A_{i,u}(t,s)
&=
\mathbb{E}[ B(s) - B(t) \, | \, Z(t) = i, U(t) = u] \\
&=
\sum_{j \in \mathcal{J}}
\int_t^s
\int_0^{u+v-t}
\bigg( b_j(v,z) + \sum_{k \in \mathcal{J}\atop k \neq j} \nu_{jk}(v,z) b_{jk}(v,z)\bigg) p_{ij}(t,u,v,\mathrm{d}z) \, \mathrm{d}v, \\
p_{ij}(t,u,s,z) &= \mathbb{P}(Z(s) = j, U(s) \leq z \, | \, Z(t) = i, U(t) = u).
\end{align*}
The transition probabilities may be calculated by solving a system of integro-differential equations, confer with~\cite[Section~3]{BuchardtMollerSchmidt}. Numerical methods for integro-differential equations can, generally speaking, be rather intricate. The implementation of semi-Markov models is, therefore, non-trivial and may carry some operational risk.

\subsubsection*{Aggregate Markov models}

In this paper, we introduce a class of aggregate models that, similar to semi-Markov models, allow for added flexibility, such as duration dependence, but avoid some of the aforementioned numerical challenges posed by semi-Markov modelling.

Denote by $(T_n)_{n\in\mathbb{N}_0}$ the jump times of $Z$, where we employ the convention $T_0 = 0$. Returning to the case where $Z$ is Markovian with transition rates $\nu_{jk}(t)$, recall that
\begin{align*}
\Prob(T_{n+1} > t \, | \, T_0, Z(T_0), T_1, Z(T_1) \ldots, T_n , Z(T_n) = j)
=
e^{\int_{T_n}^t \nu_{jj}(x) \md x}, \quad t \geq T_n.
\end{align*}
We conclude that
\begin{align*}
T_{n+1} - T_n \, \Big| \, \big(T_i,Z(T_i)\big)_{i=0}^n \sim \mathrm{IPH}\!\left(1, \nu_{Z(T_n)Z(T_n)}(T_n+ \bigcdot)\right)\!.
\end{align*}
In other words, the sojourn times follow one-dimensional IPH distributions that are mostly independent of the past history of the jump process. This paper considers instead jump processes with sojourn times admitting conditional IPH distributions of general dimension. Hereby we shall be able to capture, for instance, duration dependence while avoiding the need for intricate numerical methods.

\section{Setup}\label{sec:setup}

In this section, we present the general setup of the paper. Subsection~\ref{subsec:prob_model} introduces the probabilistic model for the state of the insured, while Subsection \ref{subsec:payment_streams} introduces the payments between the insured and the insurer.

\subsection{Probabilistic model}\label{subsec:prob_model}

Similar to Subsection~\ref{sec:multi-state}, let $Z$ be a jump process governing the state of the insured, thus taking values in the finite set of (macrostates) $\mathcal{J} = \{1,2,...,J\}$, $J\in \N$. This set consists of biometric or behavioural states that are actually observed, for example active, disabled, free-policy, and dead. To allow for added flexibility, to each macrostate we may introduce additional sub-states (microstates) that are not observable.

To be specific, to each macrostate $j$, a number $d_j \geq 1$ of microstates are assigned. The resulting state space is therefore 
\begin{align*}
E=\{ \mat{\jay} = (j,\widetilde{j}) : j\in \mathcal{J}, \widetilde{j}\in \{1,2,...,d_j\}  \},
\end{align*}
and the total number of microstates is $\bar{d} = \sum_{j\in \J} d_j$. Elements of $E$ are generally denoted by bold letters such as $\mat{\jay} \in E$. Now introduce a time-inhomogeneous Markov jump process  $\vect{X}=\{\vect{X}(t)\}_{t\geq 0}=\{(X_1(t), X_2(t))\}_{t\geq 0}$ on the state space $E$ with transition intensity matrix function $\mat{M}(t)$. Then $X_1(t)$ keeps track of the macrostate, that is $Z(t) = X_1(t)$, while $X_2(t)$ identifies the current microstate contingent on the state of $X_1(t)$.

The transition intensity matrix function $\mat{M}(t)$ can be written on the following block form: 
\begin{equation}
  \mat{M}(t) =
\begin{pmatrix}
\mat{M}_{11}(t) & \mat{M}_{12}(t) & \cdots & \mat{M}_{1J}(t) \\
 \mat{M}_{21}(t) & \mat{M}_{22}(t) & \cdots & \mat{M}_{2J}(t) \\
\vdots & \vdots & \ddots & \vdots \\
 \mat{M}_{J1}(t) & \mat{M}_{J2}(t) & \cdots & \mat{M}_{JJ}(t) \\
\end{pmatrix}\!, \label{eq:Lambda}
\end{equation}
where $\mat{M}_{jj}(t)$ are sub-intensity matrix functions of dimension $d_j \times d_j$ providing transition rates between the microstates of macrostate $j$, and $\mat{M}_{jk}(t)$ are non-negative matrix functions of dimension $d_j \times d_k$ providing transition rates from microstates within macrostate $j$ to microstates within macrostate $k$.

We denote an element of $\mat{M}(t)$ by $\mu_{\mat{\jay}\mat{\kay}}(t)$, $\mat{\jay}, \mat{\kay} \in E$. The off-diagonal elements are non-negative, providing the jump rates between different states, while the diagonal equals the negative of the row sums of the off-diagonal elements. Consequently, rows all sum to zero, so $\mat{M}(t)$ is a proper transition intensity matrix function.

For simplicity, we assume that $Z(0) = X_1(0) \equiv 1$. For a full model identification, it then suffices to specify the initial distribution $\mat{\pi}_1(0)$ of $X_2(0)$ among the microstates $1,2,\ldots,d_1$. In other words, denoting the initial distribution of $\vect{X}$ by $\vect{\pi}$, we have that
\begin{align*}
\vect{\pi} = (\vect{\pi}_1(0),\bm{0}). 
\end{align*}
The column vector function
\begin{align}\label{eq:exit_rates}
\vect{m}_j(t) = -\mat{M}_{jj}(t)\vect{1}_{d_j} = \sum_{k\in \mathcal{J} \atop k\neq j} \mat{M}_{jk}(t)\vect{1}_{d_k}
\end{align}
contains the exit rate function out of macrostate $j$. This function is non-negative due to $\mat{M}_{jj}(t)$ being a sub-intensity matrix function. The last equality follows from the row sums of $\mat{M}(t)$ being zero.

In this paper, we give special attention to that case where $\mat{M}_{jk}$, $j, k \in \mathcal{J}$, $j \neq k$, is a matrix of rank one on the form
\begin{align}\label{eq:indep_cond}
\mat{M}_{jk}(t) = \vect{\beta}_{jk}(t) \vect{\pi}_k(t),
\end{align}
where $\vect{\beta}_{jk}(t)$ is a $d_j$-dimensional non-negative column vector function and $\vect{\pi}_k(t)$ is a $d_k$-dimensional non-negative row vector function with $\vect{\pi}_k(t) \vect{1}_{d_k} = 1$. Here $\vect{\beta}_{jk}(t)$ provides the vector of jump rates from the microstates of macrostate $j$ to the macrostate $k$, and $\vect{\pi}_k(t)$ denotes the initial distribution of $X_2(t)$ on $\{1,2,\ldots,d_k\}$ just after a transition of $X_1(t)$ to $k$. In this case,
\begin{align}\label{eq:exit_rates_indep}
\vect{m}_j(t) =  \sum_{k\in \mathcal{J} \atop k\neq j} \vect{\beta}_{jk}(t).
\end{align}
{
\begin{definition}\label{def:reset_prop}
If~\eqref{eq:indep_cond} holds, we say that $k$ has the reset property from state $j$. If all states $k$ have the reset property from all states $j \neq k$, we simply say \textbf{the} reset property is satisfied.
\end{definition}}
Another way of writing~\eqref{eq:indep_cond} is
\begin{align*}
\mu_{(j,\widetilde{j})(k,\widetilde{k})}(t) = \beta_{(j,\widetilde{j})k}(t)\pi_{(k,\widetilde{k})}(t), \quad k \neq j.
\end{align*}
\begin{remark}
Focusing only on the transition from macrostate $j$ to macrostate $k$, we look at the following elements of $\mat{M}(t)$:
\begin{align*}
\begin{pmatrix}
 & \vdots  & & \vdots  & \\
 \cdots & \mat{M}_{jj}(t) & \cdots & \vect{\beta}_{jk}(t)\vect{\pi}_k(t) & \cdots  \\
 & \vdots & \ddots & \vdots & \\
 \cdots & \cdots & \cdots   & \mat{M}_{kk}(t) & \cdots \\
 & \vdots & & \vdots  & 
\end{pmatrix}\!.
\end{align*}
Here $\vect{\beta}_{jk}(t)$ is a column vector function of exit rates from states $(j,\tilde{j})$ in $\{(j,1),...,(j,d_j)\}$. Therefore it seems natural to pair $j$ and $\tilde{j}$, which explains the seemingly awkward indexation of elements of $\vect{\beta}_{jk}(t)$. Since a new microstate is picked independently of $(j,\tilde{j})$ from $\{ (k,1),...,(k,d_k)\}$, records of where the process transitioned from are lost upon transition, which explains the term reset property. \demoo
\end{remark}
Let $\mathbb{F}^Z = \{\mathcal{F}^Z(t)\}_{t\geq0}$ denote the natural filtration generated by the macrostate process $Z$. Since only $Z$ is observed, the filtration $\mathbb{F}^Z$ represents the available information. We may, as previously, associate to $Z$ a multivariate counting process $N$ with components $N_{jk} = \tuborg{N_{jk}(t)}_{t\geq 0}$ given by
\begin{align*}
N_{jk}(t)=\# \tuborg{s\in (0,t] : Z(s-)=j, Z(s)=k}
\end{align*}
as well as a marked point process $(T_n, Y_n)_{n=0}^\infty$ with $T_n$ the $n$'th jump time of $Z$ and $Y_n=Z(T_n)$; we use the convention $T_0 = 0$. Disregarding null-sets, the jump process $Z$, the multivariate counting process $N$, and the marked point process $(T_n,Y_n)_{n=0}^\infty$ generate the same information.

Although the microstate process $\vect{X}$ is Markovian, this is generally not the case for the macrostate process $Z$. In this paper, we derive distributional properties of $Z$ by deriving distributional properties of the multivariate counting process $N$ and the marked point process $(T_n, Y_n)_{n=0}^\infty$. We are especially interested in the special case where the reset property {holds}. Here it turns out that $(Z,U)$ becomes Markovian, where $U$ is the duration process defined in~\eqref{eq:U_def}.

\subsection{Payments}\label{subsec:payment_streams}

Having specified the probabilistic model, we now turn our attention to the insurance contract itself. Again, we denote by $B=\{B(t)\}_{t\geq0}$ the payments, benefits less premiums, between the insured and the insurer. We suppose that $B$ takes the form prescribed in~\eqref{eq:payments_duration}. Furthermore, throughout the paper, we assume a maximal contract time $\eta>0$ such that all sojourn payment rates and transition payments are zero after time $\eta$. 

In the later stages of the paper, we add another layer of complications by turning to so-called scaled payments that appear in connection with policyholder behaviour such as free-policy conversion and stochastic retirement. To be precise, here we furthermore consider payments $B^\rho = \{B^\rho(t)\}_{t\geq0}$ given by
\begin{align*}
\md B^\rho(t) = \rho\big(\tau,Z(\tau-),Z(\tau)\big)^{1_{(\tau \leq t)}} \md B(t), \quad B^\rho(0) = B(0),
\end{align*}
where $\tau$ is the exercise time of some policyholder option (modelled incidentally) and $0<\rho(t,j,k)\leq1$ is a suitable regular deterministic scaling factor.

The remainder of the paper now focuses on deriving distributional properties of the macrostate process $Z$, establishing computational schemes for relevant expected accumulated cash flows and prospective reserves, and finally relating these findings to existing models and methods in the life insurance literature.

\section{Properties of $Z$}\label{sec:propZ}

In this section, we derive some distributional properties of $Z$. In Subsection~\ref{subsec:general_prop}, we consider the general setup and derive the conditional finite-dimensional distributions of the marked point process $(T_n,Y_n)_{n\in \N_0}$ associated to $Z$ as well as the predictable compensators of the multivariate counting processes $N$ associated to $Z$. In Subsection~\ref{subsec:semi-Markov}, we impose the reset property, which we show, by applying the results of Subsection~\ref{subsec:general_prop}, leads to $(Z,U)$ being Markovian.

\subsection{General results}\label{subsec:general_prop}

Since $Z$ is generally not Markovian, we introduce
\begin{align*}
S_n = (T_0, Y_0, T_1, Y_1, \ldots, T_n, Y_n)
\end{align*}
to keep track of the history of $Z$. Write
\begin{align*}
s_n = (0,1,t_1,y_1,\ldots,t_n,y_n), \quad y_i  \in \mathcal{J}, 0 < t_1 < t_2 < \ldots < t_n < \infty,
\end{align*}
for a generic realization of $S_n$ whenever $T_n < \infty$. Let
\begin{align*}
\bar{F}^{(n+1)}(t \, | \, s_n) = \mathbb{P}(T_{n+1} > t \, | \, S_n = s_n)
\end{align*}
denote the conditional survival function of $T_{n+1}$ given $S_n$, and let
\begin{align*}
G^{(n+1)}(k \, | \, s_n, t_{n+1}) = \mathbb{P}(Y_{n+1} = k \, | \, S_n = s_n, T_{n+1} = t_{n+1})
\end{align*}
denote the conditional probability mass function of $Y_{n+1}$ given $(S_n, T_{n+1})$. These quantities determine the distribution of $Z$. The following result provides a characterization of them within our setup.
\begin{proposition}\label{thm:MPP}
The conditional finite-dimensional distributions of the marked point process $(T_n ,Y_n )_{n=1}^\infty$ are given by
\begin{align*}
&\bar{F}^{(n+1)}(t|s_n) =   \frac{\vect{\alpha}(s_n)}{\vect{\alpha}(s_n)\vect{1}_{d_{y_n}}} \Prodi_{t_n}^t \left(\mat{I}+\mat{M}_{y_n y_n}(x)\md x\right)\vect{1}_{d_{y_n}}, \quad t \geq t_n, \\
&G^{(n+1)}(k|s_{n}, t_{n+1}) = \frac{\vect{\alpha}(s_n)\displaystyle \Prodi_{t_n}^{t_{n+1}}\!\left(\mat{I}+\mat{M}_{y_n y_n}(x)\md x\right)\mat{M}_{y_n k}(t_{n+1})\vect{1}_{d_k}  }{\vect{\alpha}(s_n)\displaystyle \Prodi_{t_n}^{t_{n+1}}\!\left(\mat{I}+\mat{M}_{y_n y_n}(x)\md x\right)\mat{m}_{y_n}(t_{n+1})  }, \quad k \neq y_n,
\end{align*}
where the $d_{y_n}$-dimensional row vector $\vect{\alpha}(s_n)$ is given by
\begin{align*}
\vect{\alpha}(s_n) = \vect{\pi}_1(0)\displaystyle\prod_{\ell = 0}^{n-1}\,\Prodi_{t_{\ell}}^{t_{\ell+1}}\!\left(\mat{I}+\mat{M}_{y_\ell y_\ell}(x)\md x\right)\mat{M}_{y_\ell y_{\ell +1}}(t_{\ell+1}).
\end{align*}
\end{proposition}
\begin{proof}
Please refer to Appendix~\ref{ap:A}.
\end{proof}
\begin{remark}\label{rmk:iph-link}
The first statement of Proposition~\ref{thm:MPP} corresponds to
\begin{align*}
T _{n+1} - T _n \, \big|\, S_n  \sim \mathrm{IPH}\left(\frac{\vect{\alpha}(S_n )}{\vect{\alpha}(S _n)\vect{1}_{d_{Y_n}}},\, \vect{M}_{Y _nY _n}(T _n+\bigcdot)\right)\!{,}
\end{align*}
{confer with~\eqref{eq:dist_prodint}.\demoo}
\end{remark}
The compensators of the multivariate counting process associated to $Z$, which also determine the distribution of $Z$, are key quantities in the context of estimation and valuation. In our setup, they take the following form.
\begin{theorem}\label{thm:compensator}
The counting process $N_{jk}$ has $(\mathbb{F}^Z,\mathbb{P})$-compensator given by $\mathrm{d}\Lambda_{jk}(t) = \lambda_{jk}(t) \, \mathrm{d}t$ with
\begin{align*}
\lambda_{jk} (t) =  \sum_{n\in \N_0}1_{\left(T_{n} ,\, T_{n+1} \right]}(t)1_{\left(Y _{n} = j\right)}\frac{\vect{\alpha}\!\left(S_{n-1} , T _n, j\right)\displaystyle \Prodi_{T _{n}}^t\!\left(\mat{I}+\mat{M}_{jj}(x)\md x\right)}{\vect{\alpha}\!\left(S_{n-1} , T _n, j\right)\displaystyle \Prodi_{T _{n}}^t\!\left(\mat{I}+\mat{M}_{jj}(x)\md x\right)\!\vect{1}_{d_j}  }\mat{M}_{jk}(t)\vect{1}_{d_k}.
\end{align*}
\end{theorem}
\begin{proof}
Please refer to Appendix~\ref{ap:A}.
\end{proof}

\subsection{Reset property and semi-Markovianity}\label{subsec:semi-Markov}

Suppose now that the reset property holds, that is $\mat{M}_{jk}(t)$ satisfies~\eqref{eq:indep_cond} for all $j \neq k$. We now make the following observations. From~\eqref{eq:indep_cond}, we see that for $k \neq j$,
\begin{align*}
\mat{\pi}_j(t)\Prodi_t^s\!\left(\mat{I}+\mat{M}_{jj}(x)\md x\right)\mat{\beta}_{jk}(s)
\end{align*}
is a $1\times 1$-dimensional matrix, implying it cancels if appearing in both the numerator and denominator of a fraction. In particular,
\begin{align}\label{eq:reset_fraction}
\frac{\vect{\alpha}(s_n)}{\vect{\alpha}(s_n)\vect{1}_{d_{y_n}}} = \frac{\mat{\pi}_{y_n}(t_n)}{\mat{\pi}_{y_n}(t_n)\vect{1}_{d_{y_n}}} = \mat{\pi}_{y_n}(t_n).
\end{align} 
Combined with the results of Subsection~\ref{subsec:general_prop}, this yields the following corollaries.
\begin{corollary}\label{cor:MPP_reset}
Assume that {the reset property} holds. Then the conditional finite-dimensional distributions of the marked point process $(T_n ,Y_n )_{n=1}^\infty$ are given by
\begin{align*}
&\bar{F}^{(n+1)}(t|s_n) =   \vect{\pi}_{y_n}(t_n) \Prodi_{t_n}^t \left(\mat{I}+\mat{M}_{y_n y_n}(x)\md x\right)\vect{1}_{d_{y_n}}, \quad t \geq t_n, \\
&G^{(n+1)}(k|s_{n}, t_{n+1})  = \frac{\vect{\pi}_{y_n}(t_n)\displaystyle \Prodi_{t_n}^{t_{n+1}}\!\left(\mat{I}+\mat{M}_{y_n y_n}(x)\md x\right)\mat{\beta}_{y_n y_{n+1}}(t_{n+1})   }{\vect{\pi}_{y_n}(t_n)\displaystyle \Prodi_{t_n}^{t_{n+1}}\!\left(\mat{I}+\mat{M}_{y_n y_n}(x)\md x\right)\mat{m}_{y_n}(t_{n+1})  }, \quad k \neq y_n,
\end{align*}
where $\vect{m}_{y_n}(t_{n+1})$ is obtained from~\eqref{eq:exit_rates_indep}. 
\end{corollary}
\begin{remark}
The first statement of Corollary~\ref{cor:MPP_reset} corresponds to
\begin{align*}
T _{n+1}-T_n \, \big|\, S _n \sim \mathrm{IPH}\big(\vect{\pi}_{Y _n} \!\left(T _n\right)\!,\, \vect{M}_{Y _nY _n}(T _n+\bigcdot)\big). \demooeq
\end{align*}
\end{remark}
\begin{corollary}\label{cor:semi-Markov}
Assume that {the reset property} holds. Then the counting process $N_{jk}$ has $(\mathbb{F}^Z,\mathbb{P})$-compensator given by $\mathrm{d}\Lambda_{jk}(t) = \lambda_{jk}(t) \, \mathrm{d}t$ with 
\begin{align*}
\lambda_{jk} (t) &=  \sum_{n\in \N_0}1_{\left(T_{n} ,\, T_{n+1} \right]}(t)1_{\left(Y _{n} = j\right)}\frac{\vect{\pi}_j\!\left(T_{n} \right)\displaystyle \Prodi_{T _{n}}^t\!\left(\mat{I}+\mat{M}_{jj}(x)\md x\right)}{ \vect{\pi}_j\!\left(T_{n} \right)\displaystyle \Prodi_{T _{n}}^t\!\left(\mat{I}+\mat{M}_{jj}(x)\md x\right)\!\vect{1}_{d_j}  }\vect{\beta}_{jk}(t).
\end{align*}
\end{corollary}
The results show that the general path dependence of $Z$ through $\vect{\alpha}(S_n)$ is significantly reduced whenever {the reset property} is imposed. We may actually write
\begin{align*}
\lambda_{jk} (t) = 1_{(Z(t-) = j)}\frac{\vect{\pi}_j\!\left(t-U(t-) \right)\displaystyle \Prodi_{t-U(t-)}^t\!\left(\mat{I}+\mat{M}_{jj}(x)\md x\right)}{ \vect{\pi}_j\!\left(t-U(t-) \right)\displaystyle \Prodi_{t-U(t-)}^{{t}}\!\left(\mat{I}+\mat{M}_{jj}(x)\md x\right)\!\vect{1}_{d_j}  }\vect{\beta}_{jk}(t),
\end{align*}
which shows that the macrostate process $Z$ is a time-inhomogeneous semi-Markov process with transition rates $\nu_{jk}(t,u)$, $j \neq k$, which are functions of both time and duration, given by
\begin{align}\label{eq:int_reset}
\nu_{jk}(t,u)
=
\frac{\vect{\pi}_j(t-u)\displaystyle \Prodi_{t-u}^t\!\left(\mat{I}+\mat{M}_{jj}(x)\md x\right)}{ \vect{\pi}_j(t-u)\displaystyle \Prodi_{t-u}^t\!\left(\mat{I}+\mat{M}_{jj}(x)\md x\right)\!\vect{1}_{d_j}  }\vect{\beta}_{jk}(t).
\end{align}
Conversely, the class of aggregate Markov models is quite flexible. In light also of the many prized denseness results for phase type distributions, we therefore {suggest the following conjecture:
\begin{conjecture}\label{conj:dense}
The class of aggregate Markov models with the reset property is dense in the class of time-inhomogeneous semi-Markovian models with absolutely continuous sojourn time distributions.
\end{conjecture}
In other words, we conjecture that} any (smooth) semi-Markovian model can be approximated arbitrarily well by an aggregate Markov model with the reset property simply by letting the number of microstates increase to infinity. The clarification of this conjecture is outside the scope of this paper, but it nevertheless constitutes an interesting research direction.

\section{Valuation}\label{sec:reserves_cf}
In this section, we consider the valuation of the life insurance liabilities corresponding to the payment process $B$. In Subsection~\ref{subsec:valuation}, we provide expressions for the expected accumulated cash flows and, hereby, the prospective reserves. In particular, we provide matrix representations that are useful in implementing the models. The expected accumulated cash flows are composed of conditional occupation probabilities, for which we derive formulas in Subsection~\ref{subsec:probs}. Special emphasize is given to the semi-Markovian case of Subsection~\ref{subsec:semi-Markov}. Finally, in Subsection~\ref{subsec:dur_indep} and Subsection~\ref{subsec:extensions}, we investigate the impact of duration-independent payments and the inclusion of policyholder options, respectively.

Throughout the section, the time value of money is described by a deterministic and suitably regular interest rate $r(t)$. As long as financial and insurance risks are assumed independent, the extension to stochastic interest rates is straightforward.

\subsection{General results}\label{subsec:valuation}

The expected accumulated cash flow $A(t,s)$ valued at time $t$ is given by,
\begin{align*}
A(t,s) = \mathbb{E}\big[ B(s) - B(t) \, \big| \, \mathcal{F}^Z(t) \big], \quad s \geq t,
\end{align*}
confer with Definition 2.2 in \cite{buchardtfurrersteffensen2019}, so that the prospective reserve reads
\begin{align}\label{eq:reserve}
V(t)
=
\mathbb{E}\bigg[\int_t^{\eta} e^{-\int_t^s r(v) \, \mathrm{d}v} \, \mathrm{d}B(s) \, \bigg| \, \mathcal{F}^Z(t)\bigg]
=
\int_t^\eta e^{-\int_t^s r(v) \, \mathrm{d}v} A(t,\mathrm{d}s).
\end{align}
If the payments are on the form~\eqref{eq:payments_duration}, then
\begin{align*}
A(t,\mathrm{d}s)
=
\sum_{\mat{\jay} \in E} \int_0^{U(t)+s-t} p_{\mat{\jay}}(t,s,\mathrm{d}z) \bigg( b_j(s,z) +  \sum_{\mat{\kay} \in E \atop k \neq j}  b_{jk}(s,z) \mu_{\mat{\jay}\mat{\kay}}(s) \bigg) \mathrm{d}s,
\end{align*}
where the conditional occupation probabilities are given by
\begin{align}\label{eq:p_general}
p_{\mat{\jay}}(t,s,z) = \mathbb{P}\big(\vect{X}(s) = \mat{\jay}, U(s) \leq z \, \big| \, \mathcal{F}^Z(t)\big).
\end{align}
If the reset {property is} also satisfied, then $Z$ is a time-inhomogeneous semi-Markovian process and thus $A(t,s) = \sum_{i \in \mathcal{J}} 1_{(Z(t) = i)} A_{i,U(t)}(t,s)$, where
\begin{align*}
A_{i,u}(t,s)
=
\mathbb{E}\big[ B(s) - B(t) \, \big| \, Z(t) =i, U(t) = u \big].
\end{align*}
Furthermore, it holds that
\begin{align*}
A_{i,u}(t,\mathrm{d}s)
=
\sum_{\mat{\jay} \in E} \int_0^{u+s-t} p_{i\mat{\jay}}(t,u,s,\mathrm{d}z) \bigg( b_j(s,z) + \sum_{k \in \mathcal{J} \atop k \neq j} b_{jk}(s,z) \beta_{\mat{\jay}k}(s) \bigg) \mathrm{d}s,
\end{align*}
where the transition probabilities are given by
\begin{align}\label{eq:semi-markov_probs}
p_{i\mat{\jay}}(t,u,s,z) = \mathbb{P}\big(\vect{X}(s) = \mat{\jay}, U(s) \leq z \, \big| \, Z(t) = i, U(t) = u).
\end{align}
For implementation purposes, it may be beneficial to use matrix representations of the expected accumulated cash flow $A(t,s)$ following along the lines of~\cite{Bladt2020}, since it allows for more compact and direct computations.

In the general case where the aforementioned reset property is not satisfied, the process $Z$ is non-Markovian, so it is not sensible to form a transition probability matrix function in the usual way. Instead, we form a $\bar{d}$-dimensional vector function according to
\begin{align}\label{eq:p_matrix_general}
\mat{p}(t,s,\mathrm{d}z) = \left\{p_{\mat{\jay}}(t,s,\mathrm{d}z)\right\}_{\mat{\jay}\in E}.
\end{align}
In regards to payments and transition rates, however, the fact that $\vect{X}$ is assumed to be Markovian allows us to follow more closely the approach of~\cite{Bladt2020}. In the present setup, we have a set of sojourn payment rates and transition payments that are all identical across microstates (of the same macrostate). Hence, the $\bar{d}$-dimensional vector of sojourn payment rates on the micro level is given by
\begin{align}\label{eq:b}
\vect{b}(t,u) = (\vect{b}_{1}(t,u),...,\vect{b}_{J}(t,u) ), 
\end{align}
where $\vect{b}_{j}(t,u) =b_{j}(t,u) \vect{1}_{d_j}$. The matrices of transition payments must, in a similar fashion, be identical across microstates (of the same macrostate), so that the transition payment matrix function on a micro level is given by
\begin{align}\label{eq:B}
  \mat{B}(t,u) &=
\begin{pmatrix}
\mat{B}_{11}(t,u) & \mat{B}_{12}(t,u) & \cdots & \mat{B}_{1J}(t,u) \\
 \mat{B}_{21}(t,u) & \mat{B}_{22}(t,u) & \cdots & \mat{B}_{2J}(t,u) \\
 \vdots & \vdots & \ddots & \vdots \\
 \mat{B}_{J1}(t,u) & \mat{B}_{J2}(t,u) & \cdots & \mat{B}_{JJ}(t,u) \\
\end{pmatrix}\!, 
\end{align}
where $\mat{B}_{ij}(t,u)$, $i,j\in \J$, $j\neq i$, is a $d_i\times d_j$-dimensional matrix with $b_{ij}(t,u)$ in all entries, and $\mat{B}_{ii}(t,u) = \mat{0}$ is a $d_i\times d_i$-dimensional matrix of zeroes. Based hereon, we define the reward matrix function as
\begin{align}\label{eq:def_reward}
\mat{R}(t,u) = \mat{\Delta}(\vect{b}(t,u)) +  \mat{M}(t)\bullet \mat{B}(t,u),
\end{align}
where $\bullet$ denotes the Schur product, that is $(\mat{A} \bullet \mat{B})_{ij} = A_{ij}B_{ij}$, and $\mat{\Delta}(\vect{b})$ is a diagonal matrix with the vector $\vect{b}$ as diagonal. This is similar to equations (3.8)--(3.11) in~\cite{Bladt2020}. 

The expected accumulated cash flow $A(t,s)$ may then be seen to have the following matrix representation: 
\begin{align}\label{eq:Matrix_form}
A(t,\mathrm{d} s) &= \int_0^{U(t)+s-t} \mat{p}(t,s,\mathrm{d} z)\mat{R}(s,z)\vect{1}_{\bar{d}} \, \mathrm{d} s,
\end{align}
where the original sums over the state space $E$ are reduced to matrix multiplications. 
 
In the case of the reset {property,} the semi-Markovianity of $Z$ implies that it suffices to consider the transition probabilities of~\eqref{eq:semi-markov_probs}. Thus, it is sensible to form a $J \times \bar{d}$-dimensional matrix function according to
\begin{align*}
\mat{p}(t,u,s,\mathrm{d}z) = \{p_{i\mat{\jay}}(t,u,s,\mathrm{d}z)\}_{i\in \J,\mat{\jay}\in E}.
\end{align*}
Similarly, we may form the $J$-dimensional vector
\begin{align*}
\mat{A}_{u}(t,s) &= \left(A_{1,u}(t,s),\ldots,A_{J,u}(t,s)\right)
\end{align*}
of state-wise expected accumulated cash flows, which then can be calculated as follows:
\begin{align}\label{eq:Matrix_form_semi-Markov}
\mat{A}_u(t,\mathrm{d} s) = \int_0^{u+s-t} \mat{p}(t,u,s,\mathrm{d} z)\mat{R}(s,z)\vect{1}_{\bar{d}}\,\mathrm{d} s,
\end{align}
where the reward matrices $\mat{R}(t,u)$ of~\eqref{eq:def_reward} are modified according to~\eqref{eq:indep_cond}--\eqref{eq:exit_rates_indep}. We can also cast~\eqref{eq:Matrix_form_semi-Markov} as $\mat{A}_u(t,\mathrm{d} s) = \vect{a}_u(t,s) \, \mathrm{d}s$, where then
\begin{align*}
\vect{a}_u(t,s) = \int_0^{u+s-t} \ \mat{p}(t,u,s,\mathrm{d} z)\mat{R}(s,z)\vect{1}_{\bar{d}}
\end{align*}
is a vector of state-wise expected cash flows.

\subsection{Conditional occupation and transition probabilities}\label{subsec:probs}

We now provide calculation schemes for the conditional occupation and transition probabilities. Rather than working directly with these quantities, it turns out to be fruitful to focus instead on
\begin{align*}
\bar{p}_{\mat{\jay}}(t,s,z) &= \mathbb{P}\big(\vect{X}(s) = \mat{\jay}, U(s) > z \, \big| \, \mathcal{F}^Z(t)\big), \\
\bar{p}_{i\mat{\jay}}(t,u,s,z) &= \mathbb{P}\big(\vect{X}(s) =\mat{\jay}, U(s) > z \, \big| \, Z(t) = i, U(t) = u),
\end{align*}
which suffices since
\begin{align*}
\bar{p}_{\mat{\jay}}(t,s,\mathrm{d} z) &= - p_{\mat{\jay}}(t,s,\mathrm{d} z), \\
\bar{p}_{i\mat{\jay}}(t,u,s,\mathrm{d} z) &= - p_{i\mat{\jay}}(t,u,s,\mathrm{d} z).
\end{align*}
In the following, we require the $\bar{d}\times d_j$-dimensional matrices
\begin{align*}
\mat{E}_j = \sum_{\widetilde{j} = 1}^{d_j} \vect{e}_{\mat{\jay}}\vect{e}_{\widetilde{j}}^{\prime},
\end{align*}
where $\vect{e}_{\widetilde{j}}$ is a $d_j$-dimensional column with a one in entry $\widetilde{j}$ and otherwise zeroes, and $\vect{e}_{\mat{\jay}}$ is a $\bar{d}$-dimensional column vector with a one in entry $d_1 + \cdots + d_{j-1} + \widetilde{j}$ and otherwise zeroes. Here and in the following, primes denotes matrix transposition. The entries of $\mat{E}_j$ are zero, except in the $j$'th block row, where they consist of the $d_j$-dimensional identity matrix. Roughly speaking, they allow us to extend a distribution on microstates in a single macrostate to the whole state space $E$ (and vice versa).
\begin{theorem}\label{thm:p}
It holds that
\begin{align*}
\bar{p}_{\mat{\jay}}(t,s,z) = 
\sum_{n\in\mathbb{N}_0}
1_{\left[T_{n} ,\, T_{n+1} \right)}(t)\widetilde{p}_{\mat{\jay}}(t,s,z; S_{n}),
\end{align*}
where the auxiliary quantities $\widetilde{p}_{\mat{\jay}}(t,s,z; s_n)$ are zero for $t_n \geq s - z$ and
\begin{align*}
&\widetilde{p}_{\mat{\jay}}(t,s,z; s_{n}) \\
&=
\frac
{
\vect{\alpha}(s_{n})\displaystyle \Prodi_{t_{n}}^t\!\left(\mat{I}+\mat{M}_{y_ny_n}(x)\mathrm{d} x\right)\mat{E}^{\prime}_{y_n}
}
{
\vect{\alpha}(s_{n})\displaystyle \Prodi_{t_{n}}^t\!\left(\mat{I}+\mat{M}_{y_ny_n}(x)\mathrm{d} x\right)\vect{1}_{d_{y_n}}
}
\Prodi_{t}^{(s-z)\vee t}\!\left(\mat{I}+\mat{M}(x)\mathrm{d} x\right)\!\mat{E}_j\!
\Prodi_{(s-z)\vee t}^{s}\!\left(\mat{I}+\mat{M}_{jj}(x)\mathrm{d} x\right)\vect{e}_{\widetilde{j}}
\end{align*}
for $t_n < s - z$.
\end{theorem}
Note that
\begin{align}\label{eq:yn_is_j}
1_{(t_n < s-z \leq t)}
\widetilde{p}_{\mat{\jay}}(t,s,z; s_{n})
=
1_{(y_n = j)}
\frac
{
\vect{\alpha}(s_{n})\displaystyle \Prodi_{t_{n}}^s\!\left(\mat{I}+\mat{M}_{y_ny_n}(x)\mathrm{d} x\right)\vect{e}_{\widetilde{j}}
}
{
\vect{\alpha}(s_{n})\displaystyle \Prodi_{t_{n}}^t\!\left(\mat{I}+\mat{M}_{y_ny_n}(x)\mathrm{d} x\right)\vect{1}_{d_{y_n}}
}.
\end{align}
\begin{proof}
Please refer to Appendix~\ref{ap:A}.
\end{proof}
If the reset {property is} satisfied, in which case $Z$ is a time-inhomogeneous semi-Markovian process, we can use~\eqref{eq:reset_fraction} to immediately obtain the following corollary.
\begin{corollary}\label{cor:semi_markov_prob}
Assume {that the reset property} holds. Then $\bar{p}_{i\mat{\jay}}(t,u,s,z)$ is zero for $t-u \geq s - z$ and
\begin{align*}
&\bar{p}_{i\mat{\jay}}(t,u,s,z) \\
&=
\frac{\vect{\pi}_i(t-u)\displaystyle \Prodi_{t-u}^t\!\left(\mat{I}+\mat{M}_{ii}(x)\mathrm{d} x\right)\mat{E}^{\prime}_i }{\vect{\pi}_i(t-u)\displaystyle \Prodi_{t-u}^t\!\left(\mat{I}+\mat{M}_{ii}(x)\mathrm{d} x\right)\vect{1}_{d_i}}\Prodi_{t}^{(s-z)\vee t}\!\left(\mat{I}+\mat{M}(x)\mathrm{d} x\right)\!\mat{E}_j\!
\Prodi_{(s-z)\vee t}^{s}\!\left(\mat{I}+\mat{M}_{jj}(x)\mathrm{d} x\right)\vect{e}_{\widetilde{j}}
\end{align*}
for $t - u < s - z$.
\end{corollary}
We note that $z \mapsto \bar{p}_{i\mat{\jay}}(t,u,s,z)$ is continuous on $[0,u+s-t)$ and actually constant on $[s-t, u+s-t)$. If $i \neq j$, then the continuity extends to $[0,u+s-t]$, while
\begin{align*}
\Delta \bar{p}_{i\mat{\iay}}(t,u,s,u+s-t)
=
- \lim_{h \downarrow 0} \bar{p}_{i\mat{\iay}}(t,u,s,s-t-h).
\end{align*}
The fact that $z \mapsto \bar{p}_{i\mat{\jay}}(t,u,s,z)$ is constant on $[s-t, u+s-t)$ may be utilized to reduce the computational load when calculating the expected accumulated cash flows.

We conclude this subsection by presenting an algorithm for the computation of expected cash flows in models with the reset property{, see the following page}. {The computational} scheme is similar to the algorithm for general semi-Markov models based on Kolmogorov's forward integro-differential equation proposed in~\cite[Section~3]{BuchardtMollerSchmidt}. Both algorithms employ a two-dimensional time and duration grid, and one would therefore expect the computational loads to be comparable.

\subsection{Duration-independent payments}\label{subsec:dur_indep}

We now consider the simplifications arising from duration-independent payments, that is, when
\begin{align}\label{eq:no_dur_pay}
b_j(t,u) = b_j(t) \quad &\text{and} \quad b_{jk}(t,u) = b_{jk}(t),
\end{align}
or, equivalently,
\begin{align*}
\vect{b}(t,u) = \vect{b}(t)\quad &\text{and}\quad \mat{B}(t,u) = \mat{B}(t).
\end{align*}
In this case,
{
\begin{align*}
A(t,\mathrm{d}s)
=
\sum_{\mat{\jay}\in E} \bar{p}_{\mat{\jay}}(t,s,0) \bigg( b_j(s) +  \sum_{\mat{\kay}\in E \atop k \neq j}  b_{jk}(s) \mu_{\mat{\jay}\mat{\kay}}(s) \bigg) \mathrm{d}s
=
\mat{\bar{p}}(t,s,0)\mat{R}(s)\vect{1}_{\bar{d}} \, \mathrm{d}s,
\end{align*}}
where 
\begin{align*}
\mat{\bar{p}}(t,s,0) &= \left\{\bar{p}_{\mat{\jay}}(t,s,0)\right\}_{\mat{\jay}\in E}, \\
\mat{R}(t) &= \mat{\Delta}(\vect{b}(t)) +  \mat{M}(t)\bullet \mat{B}(t).
\end{align*}
According to Theorem~\ref{thm:p},
\begin{align}\nonumber
\bar{p}_{\mat{\jay}}(t,s,0) &= 
\sum_{n\in\mathbb{N}_0}
1_{\left[T_{n} ,\, T_{n+1} \right)}(t)\widetilde{p}_{\mat{\jay}}(t,s,0; S_{n}), \\ \label{eq:nice_moments}
\widetilde{p}_{\mat{\jay}}(t,s,0; s_{n})
&=
\frac
{
\vect{\alpha}(s_{n})\displaystyle \Prodi_{t_{n}}^t\!\left(\mat{I}+\mat{M}_{y_ny_n}(x)\mathrm{d} x\right)\mat{E}^{\prime}_{y_n}
}
{
\vect{\alpha}(s_{n})\displaystyle \Prodi_{t_{n}}^t\!\left(\mat{I}+\mat{M}_{y_ny_n}(x)\mathrm{d} x\right)\vect{1}_{d_{y_n}}
}
\Prodi_{t}^{s}\!\left(\mat{I}+\mat{M}(x)\mathrm{d} x\right)\!\vect{e}_{\mat{\jay}}.
\end{align}

\begin{algorithm}[H]
\caption{Computation of expected cash flows in an aggregate Markov model with the reset property.}\label{alg:cf_reset}
\begin{algorithmic}
\State \textit{\textbf{Input}}: Current time $t \in [0,\eta)$, current duration $u\in [0,t]$, and a grid $\mathcal{T}:\, t = t_0<t_1<\cdots<t_n = \eta$ on the interval $[t,\eta]$.
\begin{enumerate} 
\item[ 1)] 
Calculate initial conditional distributions at time $t$:
\begin{align*}
\vect{\gamma}_i(t,u) &= \frac{\vect{\pi}_i(t-u)\displaystyle \Prodi_{t-u}^t\!\left(\mat{I}+\mat{M}_{ii}(x)\mathrm{d} x\right)\mat{E}^{\prime}_i }{\vect{\pi}_i(t-u)\displaystyle \Prodi_{t-u}^t\!\left(\mat{I}+\mat{M}_{ii}(x)\mathrm{d} x\right)\vect{1}_{d_i}}, \qquad i\in \mathcal{J}, \\[0.4 cm]
\vect{\gamma}(t,u) &= \left(\vect{\gamma}_1(t,u), \ldots, \vect{\gamma}_J(t,u)\right)'.
\end{align*} 
\item[ 2)] Compute transition probabilities for the Markovian state process $\vect{X}$,
\begin{align*}
\mat{P}(t,t_\ell) = \Prodi_{t}^{t_\ell}\!\left(\mat{I}+\mat{M}(x)\mathrm{d} x\right), \qquad \ell \in \{1,\ldots,n\},
\end{align*}
by solving Kolmogorov's forward differential equation on $\mathcal{T}$. \medskip 
\item[ 3)] For $\ell \in \{1,\ldots,n\}$: 

\begin{enumerate}[i)]
\item Compute state-wise stay probabilities until time $t_\ell$:  
\begin{align*}
\mat{\bar{P}}_{jj}(t_{\ell^{\prime}},t_\ell) &= \Prodi_{t_{\ell^{\prime}}}^{t_\ell}\!\left(\mat{I}+\mat{M}_{jj}(x)\mathrm{d} x\right), \qquad t_{\ell^{\prime}}\in \mathcal{T},\ \ell^{\prime} \leq \ell,\ j\in \mathcal{J}, \\
\mat{\bar{P}}(t_{\ell^{\prime}},t_\ell) &= \mat{\Delta}\!\left((\mat{\bar{P}}_{11}(t_{\ell^{\prime}},t_\ell),\ldots,\mat{\bar{P}}_{JJ}(t_{\ell^{\prime}},t_\ell)\right)\!, 
\end{align*}
by solving Kolmogorov's backward differential equation on $\mathcal{T}$ starting at  $t_\ell$.\medskip 
\item Calculate the vector of state-wise expected cash flows for time $t_\ell$: 
\begin{align*}
\vect{a}_u(t,t_\ell) &= \vect{\gamma}(t,u)\bigg(\int_t^{t_{\ell}} \mat{P}(t,v)\mat{\widetilde{M}}(v)\mat{\bar{P}}(v,t_\ell)\mat{R}(t_\ell,t_\ell - v)\vect{1}_{\bar{d}} \,\, \mathrm{d}v \\
& \qquad\qquad\qquad + \mat{\bar{P}}(t,t_\ell)\mat{R}(t_\ell, u+t_\ell - t)\vect{1}_{\bar{d}} \bigg),
\end{align*}
using numerical integration methods on the grid $\mathcal{T}$ for the integral, and where
\begin{align*}
\widetilde{\mat{M}}_j(v) &= \mat{M}(v)\mat{E}_j - \mat{E}_j\mat{M}_{jj}(v),\qquad j\in \mathcal{J}, \\[0.2 cm]
\widetilde{\mat{M}}(v) &= \left(\widetilde{\mat{M}}_1(v), \ldots, \widetilde{\mat{M}}_J(v)\right)\!.
\end{align*}
\end{enumerate}
\end{enumerate}
    \State \textit{\textbf{Output}:\ For each  $\ell \in \{1,\ldots,n\}$, a vector of state-wise expected cash flows\, $\vect{a}_u(t,t_\ell)$}. 
\end{algorithmic}
\end{algorithm}

If the reset {property is} satisfied, then we are rather interested in $\vect{A}_u(t,s)$, which subject to~\eqref{eq:no_dur_pay} reads
\begin{align*}
\mat{A}_u(t,\mathrm{d} s) = \mat{\bar{p}}(t,u,s,0)\mat{R}(s)\vect{1}_{\bar{d}}\,\mathrm{d} s,
\end{align*}
where
\begin{align*}
\mat{\bar{p}}(t,u,s,0) &= \left\{\bar{p}_{i\mat{\jay}}(t,u,s,0)\right\}_{i\in\mathcal{J},\mat{\jay}\in E}.
\end{align*}
Furthermore,
\begin{align*}
\bar{p}_{i\mat{\jay}}(t,u,s,0)
=
\frac{\vect{\pi}_i(t-u)\displaystyle \Prodi_{t-u}^t\!\left(\mat{I}+\mat{M}_{ii}(x)\mathrm{d} x\right)\mat{E}^{\prime}_i }{\vect{\pi}_i(t-u)\displaystyle \Prodi_{t-u}^t\!\left(\mat{I}+\mat{M}_{ii}(x)\mathrm{d} x\right)\vect{1}_{d_i}}
\Prodi_{t}^{s}\!\left(\mat{I}+\mat{M}(x)\mathrm{d} x\right)\!\vect{e}_{\mat{\jay}}.
\end{align*} 
This should lead to a significant reduction in computational load since the above simplification allows one to adapt Algorithm \ref{alg:cf_reset} to employ only a one-dimensional time grid. For general semi-Markov models, where the computation of transition probabilities relies on Kolmogorov's forward integro-differential equation, such a simplification is not possible. It should be noted, however, that the computation of the term
\begin{align*}
\Prodi_{t}^{s}\!\left(\mat{I}+\mat{M}(x)\mathrm{d} x\right)
\end{align*}
might still be rather burdensome if $\bar{d}$ is large. To conclude, if the number of microstates per macrostates is not too large, aggregate Markov models might hold a competitive advantage over general semi-Markov models if the payments of interest are duration-independent.

The above discussion relates to duration-independent payments. However, it is also applicable to certain crude duration-dependent payments. This is partly illustrated by the numerical example in Section~\ref{sec:num}, where we consider a contract stipulating a waiting period.

\subsection{Policyholder behaviour}\label{subsec:extensions}

We now extend the results of Subsections~\ref{subsec:valuation}--\ref{subsec:probs} to include incidental policyholder behaviour such as free-policy conversion and expedited or postponed retirement. The inclusion of policyholder options is quite popular in the life insurance literature, confer with~\cite{henriksen2014, buchardt2015, BuchardtMollerSchmidt, GadNielsen}, especially for Markov chains. General insights based on change of measure techniques were recently provided in~\cite{furrer2022}. In the following, we adapt the general methods and results of~\cite{furrer2022} to our setting.

Suppose that the macrostates $\mathcal{J}$ can be decomposed as
\begin{align*}
\mathcal{J} = \mathcal{J}_0 \cup \mathcal{J}_1\cup \left\{\nabla\right\},
\end{align*}
with $1\in \mathcal{J}_0$ and where the transition intensity matrix function $\mat{M}(t)$ of the microstate process $\vect{X}$ is composed of block matrix functions satisfying
\begin{align*}
\mat{M}_{jk}(t) &=\mat{0}, \quad &j\in \mathcal{J}_1, k\in \mathcal{J}_0, \\ 
\mat{M}_{j\nabla}(t) &= \mat{0}, &j\in \mathcal{J}.
\end{align*} 
In that case, the macrostate process $Z$ almost surely never hits $\nabla$, and, upon entering the states $\mathcal{J}_1$, the process never returns to $\mathcal{J}_0$. Recall that $Z(0) \equiv 1$, so the process starts in $\mathcal{J}_0$. We may thus interpret $\mathcal{J}_0$ as the states of the insured prior to exercising their policyholder option and $\mathcal{J}_1$ as the states of the insured after exercising the option. (The role of the `dummy' state $\nabla$ will be clear later.) The first hitting time of $\mathcal{J}_1$, given by
\begin{align*}
\tau = \inf\{ t > 0 : Z(t) \in \mathcal{J}_1\},
\end{align*}
then constitutes the exercise time of the option. Since $Z$ almost surely never hits $\nabla$, we may as well take $b_\nabla(s,u) = 0$, $b_{j\nabla}(s,u) = 0$, and $b_{\nabla k}(s,u) =0 $. At exercise, the original contractual payments are scaled with the factor $\rho(\tau,Z(\tau-),Z(\tau))$, where $0 < \rho(t,j,k) \leq 1$ is some suitably regular deterministic function. The payment process of interest $B^\rho = \{B^\rho(t)\}_{t\geq0}$ thus takes the form
\begin{align*}
\md B^\rho(t) = \rho\big(\tau,Z(\tau-),Z(\tau)\big)^{1_{(\tau \leq t)}} \md B(t), \quad B^\rho(0) = B(0).
\end{align*}
The scaling factor is typically selected as to maintain actuarial equivalence with respect to a safe-side valuation basis, the so-called technical basis; we just consider it given. The corresponding expected accumulated cash flow $A^\rho(t,s)$ valued at time $t$ is
\begin{align*}
A^\rho(t,s)
=
\mathbb{E}[B^\rho(s)-B^\rho(t)\, | \, \mathcal{F}^Z(t)] 
=
\mathbb{E}\bigg[\int_t^s  \rho\big(\tau,Z(\tau-),Z(\tau)\big)^{1_{(\tau \leq u)}} \mathrm{d}B(u) \,\bigg| \, \mathcal{F}^Z(t)\bigg].
\end{align*}
The following result is a consequence of~\cite[Theorem~3.6 and Proposition~3.10]{furrer2022}.
\begin{proposition}\label{prop:measurechange}
It holds that
\begin{align*}
A^\rho(t,s)
=
\widehat{\mathbb{E}}[B(s) - B(t) \, | \, \mathcal{F}^Z(t)] \rho\big(\tau,Z(\tau-),Z(\tau)\big)^{1_{(\tau \leq t)}},
\end{align*}
where $\widehat{\mathbb{E}}$ denotes expectation with respect to another probability measure $\widehat{\mathbb{P}}$. Furthermore, the $(\mathbb{F}^Z, \widehat{\mathbb{P}})$-compensators of the counting processes are given by
\begin{align*}
\mathrm{d} \widehat{\Lambda} _{jk}(t) &= \rho(t,j,k) \, \mathrm{d} \Lambda _{jk}(t), \quad &j\in \mathcal{J}_0, k\in \mathcal{J}_1, \\[2mm]
\mathrm{d} \widehat{\Lambda} _{j\nabla}(t) &= \sum_{k\in \mathcal{J}_1}\!\left(1-\rho(t,j,k)\right) \mathrm{d}\Lambda _{jk}(t),  &j\in \mathcal{J}_0, \\
\mathrm{d} \widehat{\Lambda} _{\nabla k}(t) &= 0, &k\in \mathcal{J}, k\neq \nabla, \\[2mm]
\mathrm{d} \widehat{\Lambda} _{jk}(t) &= \mathrm{d} \Lambda _{j k}(t), &\text{otherwise.} 
\end{align*}
\end{proposition}
Recall that the compensators determine the distribution of $Z$. Thus from the expressions for the  compensators obtained in Theorem~\ref{thm:compensator} and the above proposition, we find that $Z$ under $\widehat{\mathbb{P}}$ follows an aggregate Markov model with transition intensity matrix function $\mat{\widehat{M}}(t)$ composed of block matrix functions
\begin{align*}
\mat{\widehat{M}}_{jk}(t) &= \rho(t,j,k)\mat{M}_{jk}(t), \quad &j\in \mathcal{J}_0, k\in \mathcal{J}_1, \\
\mat{\widehat{M}}_{j\nabla}(t) &= \sum_{k\in \mathcal{J}_1}\!\left(1-\rho(t,j,k)\right)\!\mat{M}_{jk}(t)\vect{e},  &j\in \mathcal{J}_0,\\
\mat{\widehat{M}}_{\nabla k}(t) &= \mat{0}, &k\in \mathcal{J}, k\neq \nabla,\\
\mat{\widehat{M}}_{jk}(t) &= \mat{M}_{j k}(t), &\text{otherwise.} 
\end{align*}
Furthermore, if the reset {property is} satisfied under $\mathbb{P}$, this is also the case under $\widehat{\mathbb{P}}$. All in all, according to Proposition~\ref{prop:measurechange} the expected accumulated cash flow
\begin{align*}
\widehat{A}(t,s) = \widehat{\mathbb{E}}[B(s) - B(t) \, | \, \mathcal{F}^Z(t)],
\end{align*}
and thus also the expected accumulated cash flow $A^\rho(t,s)$, can be calculated using the formulas of Subsections~\ref{subsec:valuation}--\ref{subsec:probs}, but with $\mat{M}(t)$ replaced by $\mat{\widehat{M}}(t)$.

\section{Numerical example}\label{sec:num}

We conclude the paper by presenting a numerical example that serves to illustrate the methods presented in Section~\ref{sec:reserves_cf}. The probabilistic models, described by transition rates on the micro level, are taken from the numerical example in~\cite{AhmadBladt}, where aggregate Markov models corresponding to Figure~\ref{fig:disability_micro} with the reset property are fitted to simulated data on a macro level for different numbers of disability microstates, $d_2$. The simulations are based on a (smooth) semi-Markovian disability model employed by a large Danish life insurance company that has been reported to and published by the Danish Financial Supervisory Authority. The only duration effects present in this model concern the rates from the disability state, which also explains why we do not add extra microstates to the active macrostate. The rates from the disability state are, at least after some months, decreasing as functions of duration.

{
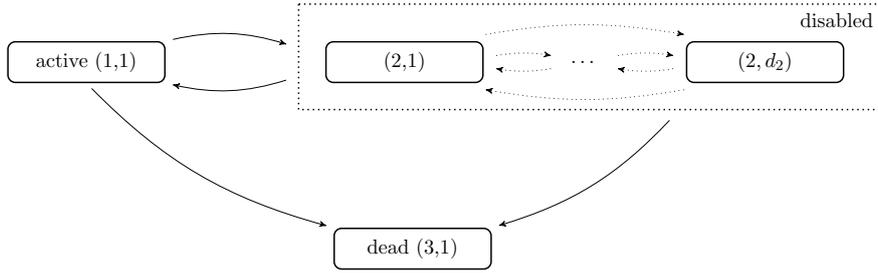
\begin{figure}[h!]
	\centering
\scalebox{0.7}{	\begin{tikzpicture}[node distance=2em and 0em]
		\node[punkt] (21) {$(2,\!1)$};
		\node[right = 15mm of 21] (2temp) {$\cdots$};
		\node[punkt, right = 38mm of 21] (22) {$(2,d_2)$};
		\node[punkt, left = 30mm of 21] (11) {active $(1,\!1)$};
		\node[draw = none, fill = none, left = 50 mm of 22] (test) {};
		\node[punkt, below = 30mm of test] (3) {dead $(3,\!1)$};
		\node[right = 15pt] at ($(22.south)+(0,1.2)$) {disabled};
		\draw[thick, dotted] ($(21.north west)+(-0.5,0.7)$) rectangle ($(22.south east)+(0.8,-0.5)$);
		\path
		($(11.north east)$)	edge [pil, bend left = 15]	($(21.north west) + (-0.6,-0.1)$)
		($(21.south west)+(-0.6,0.1)$) edge [pil, bend left=15] ($(11.south east)$)
		($(22.south) + (-1.7,-0.6)$) edge [pil, bend left=15] ($(3.north east)$)
		($(21.north east)  + (-0.1,0.1)$)	edge [pildotted, bend left = 10] ($(22.north west) + (0.1,0.1)$)
		($(22.south west)  + (0.1,-0.1)$)	edge [pildotted, bend left = 10] ($(21.south east) + (-0.1,-0.1)$)
		($(21.east)  + (0.1,0.1)$)	edge [pildotted, bend left = 10] ($(2temp.west) + (-0.1,0.1)$)
		($(2temp.west)  + (-0.1,-0.1)$)	edge [pildotted, bend left = 10] ($(21.east) + (0.1,-0.1)$)
		($(2temp.east)  + (0.1,0.1)$)	edge [pildotted, bend left = 10] ($(22.west) + (-0.1,0.1)$)
		($(22.west)  + (-0.1,-0.1)$)	edge [pildotted, bend left = 10] ($(2temp.east) + (0.1,-0.1)$)
		($(11.south)$) edge [pil, bend right=15] ($(3.north west)$)
	;
	\end{tikzpicture}}
	\caption{Disability model with $d_2$ unobservable disability microstates.}
	\label{fig:disability_micro}
\end{figure}}

{
In the following, the analysis of~\cite{AhmadBladt} is extended with calculations and comparisons of state-wise expected cash flows and prospective reserves. We focus on coverage which admits duration-dependent payments, namely disability coverage with a waiting period. To be specific, we consider a male of age $t=40$ years with a disability annuity of rate $1$ per year, starting $3$ months after the onset of disability, but only until retirement at age $65$. The only non-zero payments function is thus $b_2(s,z)$, which reads
\begin{align*}
b_2(s,z) = 1_{(s < 65)} 1_{(z > 1/4)},
\end{align*}
and we may therefore set $\eta = 65$. The aggregate Markov models fitted in~\cite{AhmadBladt} have the reset property and, consequently, the associated macrostate processes are (time-inhomogeneous) semi-Markov and with transition rates according to~\eqref{eq:int_reset}. In Figure~\ref{fig:rates_EM}, we provide for illustrative purposes a selection of the estimated transition rates that are particularly relevant for the coverage of interest, but only for initial duration $u=0,1$ years. We refer to the numerical example in~\cite{AhmadBladt} for further details.}

{
\begin{figure}[h!]
   \centering
   \includegraphics[width=0.85\textwidth,trim={8mm 4mm 8mm 0mm}]{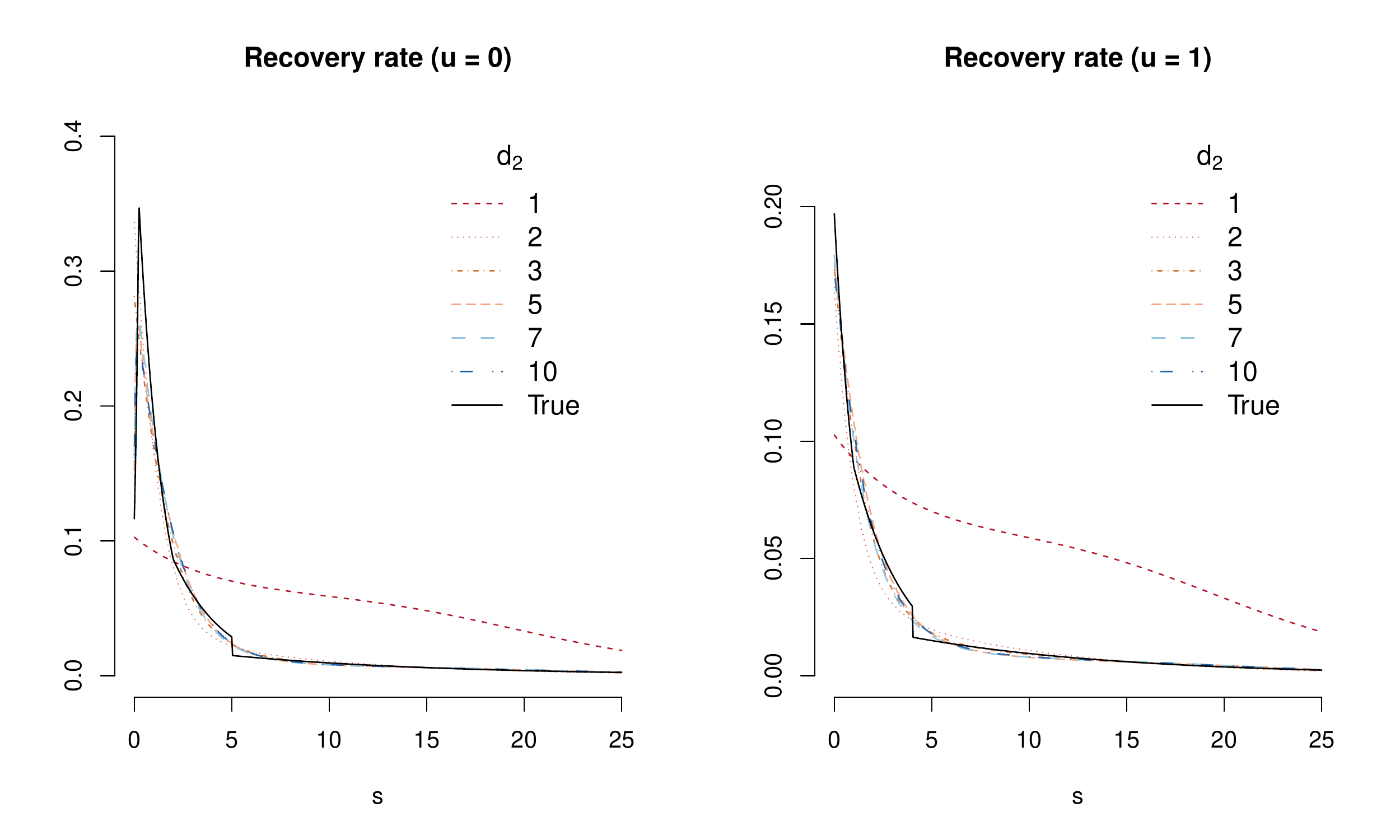}
   \includegraphics[width=0.85\textwidth,trim={8mm 4mm 8mm 0mm}]{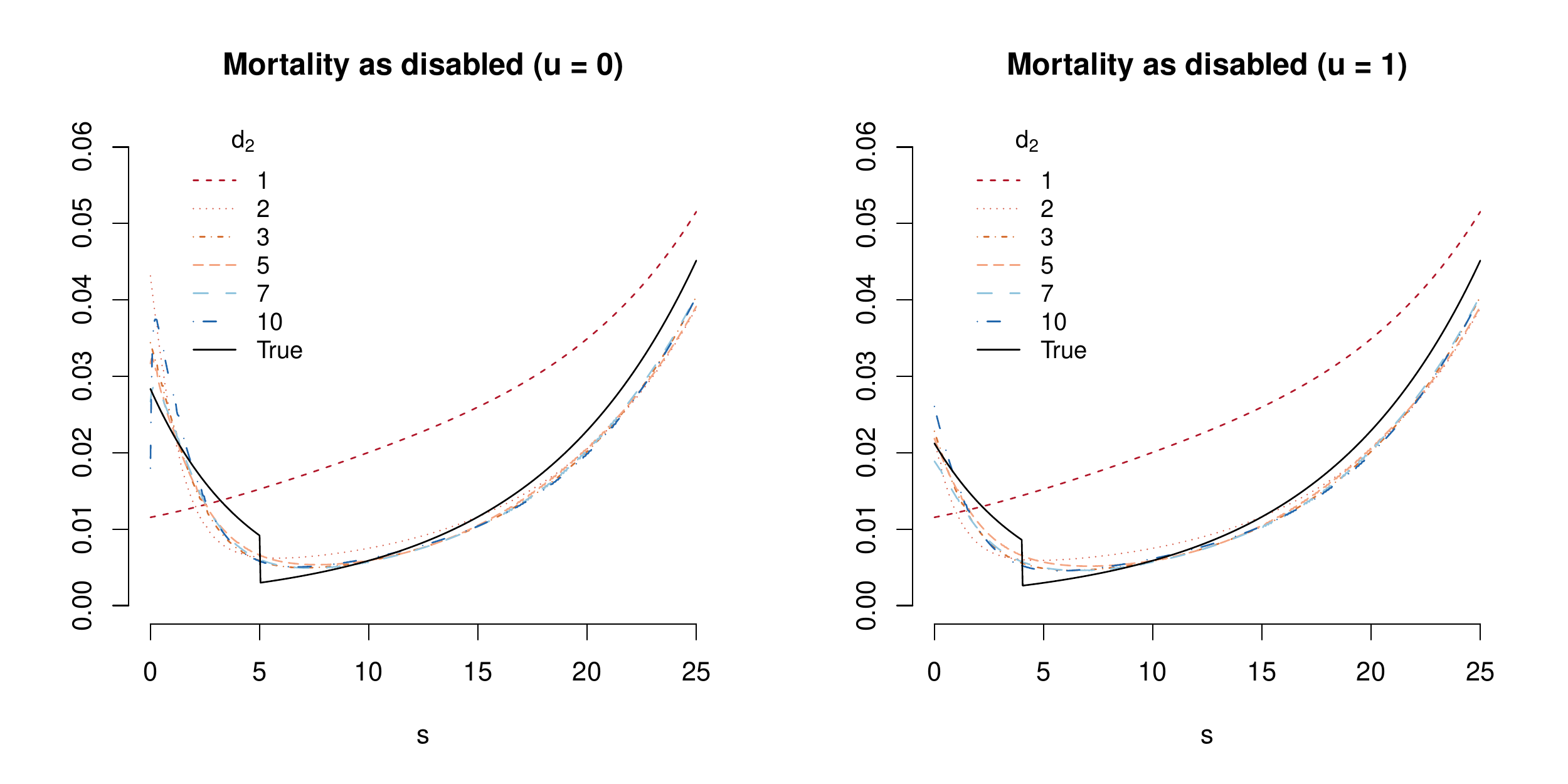} 
\caption{
{     
For $u=0,1$ illustration of the recovery rate $s \mapsto\nu_{21}(40+s,u+s)$ and the mortality as disabled $s \mapsto \nu_{23}(40+s,u+s)$ obtained from the aggregate Markov models fitted in~\cite{AhmadBladt}, confer with~\eqref{eq:int_reset}. The case $d_2 = 1$ corresponds to a Markov chain.} 
}
   \label{fig:rates_EM}
\end{figure}}

{
We emphasize that the particular simple type of duration dependence of the payments} allows for simplifications in the computation schemes similar to those from the duration-independent case of Subsection~\ref{subsec:dur_indep}. Indeed, the vector of state-wise expected cash flows now reads
\begin{align*}
\vect{a}_u(40,s) = 1_{(u+s-40>1/4)}\mat{\bar{p}}(40,{u},{s},1/4)\mat{E}_2\vect{1}_{d_2},  
\end{align*} 
where the elements of $\mat{\bar{p}}(40,s,u,1/4)$ can be calculated using Corollary~\ref{cor:semi_markov_prob}, confer also with Algorithm~\ref{alg:cf_reset}. Since we only need the transition probabilities at a single (and rather small) end duration $z=1/4$, the computational complexity is comparable with that of the duration-independent case, where only $z=0$ is needed.

The corresponding vector of state-wise prospective reserves is obtained by discounting and accumulating the vector of state-wise expected cash flows:
\begin{align*}
\mat{V}_u(40)
=
\int_{40}^{65} e^{-\int_{40}^s r(v) \, \mathrm{d}v} \vect{a}_u(40,s) \, \mathrm{d}s.
\end{align*}
For the interest rate, we use the forward rate curve published on the 3rd of November, 2022, by the Danish Financial Supervisory Authority. We are implicitly assuming that the financial market is independent of the state process of the insured.

We calculate the vectors of state-wise expected cash flows and corresponding prospective reserves across initial states and durations and for the various fits. We also calculate these quantities for the underlying true semi-Markov model, where we need to use Kolmogorov's forward integro-differential equation of~\cite[Section~3]{BuchardtMollerSchmidt}, since this model is not an aggregate Markov model. Figure~\ref{fig:cf_pfa} shows the resulting expected cash flows in the disability state for initial {duration $u=0,1$}, while Figure~\ref{fig:reserves_pfa} shows the prospective disability reserves as functions of the duration since the onset of disability. 

\begin{figure}[h!]
   \centering
   \includegraphics[width=0.85\textwidth,trim={8mm 4mm 8mm 0mm}]{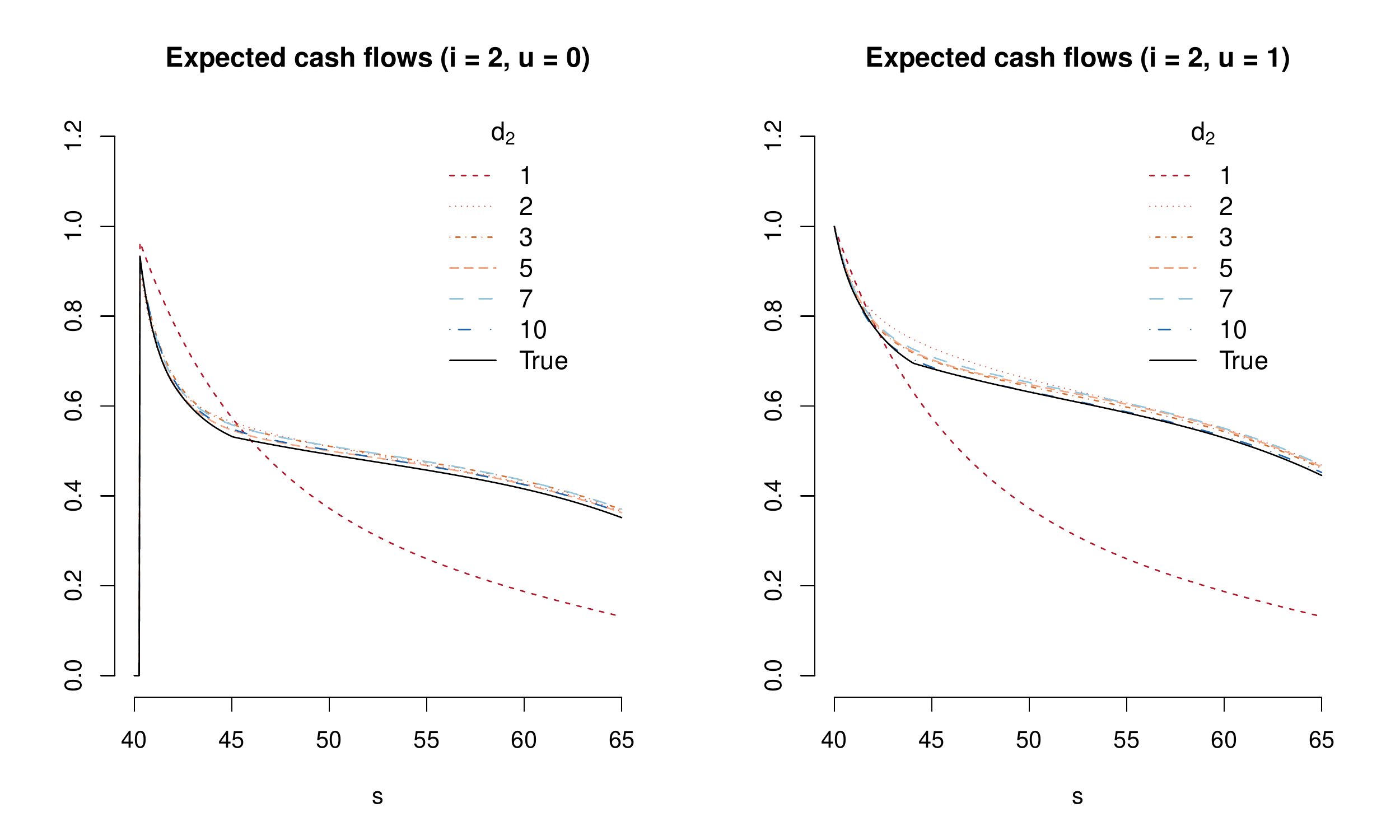}
   \caption{Expected cash flow $s\mapsto a_{2,u}(40,s)$ in the disability state with initial durations $u = 0$ (left) and $u=1$ (right) for different numbers of disability microstates, $d_2$, along with the true expected accumulated cash flows. The case $d_2 = 1$ corresponds to a Markov chain.}
   \label{fig:cf_pfa}
\end{figure}

\begin{figure}[h!]
   \centering
   \includegraphics[width=0.6\textwidth,trim={8mm 4mm 8mm 0mm}]{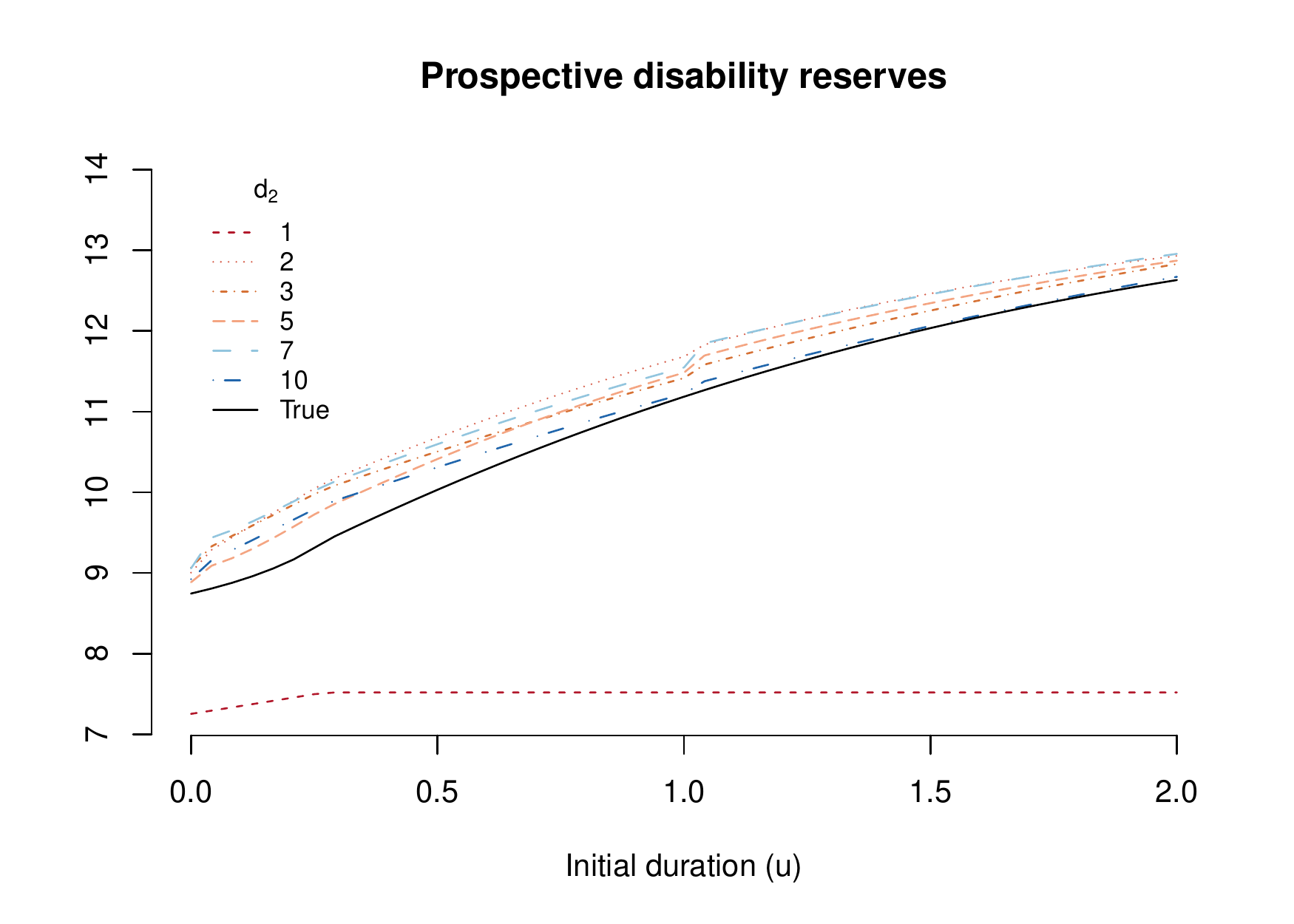}
   \caption{Prospective disability reserve as a function of duration since onset of disability, $u\mapsto V_{2,u}(40)$, for different numbers of disability microstates, $d_2$, along with true prospective reserve. The case $d_2 = 1$ corresponds to a Markov chain.}
   \label{fig:reserves_pfa}
\end{figure}

Since it is unable to capture the duration effects that are present, the Markov chain corresponding to $d_2 = 1$ performs, as anticipated, very badly. Maybe more surprisingly, the addition of just a single additional disability microstate corresponding to $d_2 = 2$ leads to significant improvements; this model may already be competitive, depending also on the trade-off between accuracy and computational load. Furthermore, and consistent with our expectations, the accuracy appears to further improve as the number of disability microstates, $d_2$, increases. {The latter point might be less clear from a purely visual standpoint, especially for $d_2 = 7$, but it is confirmed by the final log-likelihood values from the statistical analyses in~\cite{AhmadBladt}, confer also with Table~\ref{tab:log-liks}.

\input{log_liks.tex}

Another interesting observation is the differences in accuracy across regions of initial duration, $u$. For instance, while the model corresponding to $d_2 = 3$ outperforms the model corresponding to $d_2 = 5$ for $u \in [1,2]$, it is far from competitive for smaller $u$. As discussed in~\cite{AhmadBladt}, this is mainly due to the irregular (non-smooth) nature of the true transition rates in that region, see also Figure~\ref{fig:rates_EM}.}

{\section{Concluding remarks}\label{sec:conc}

In this paper, we introduce a new class of models for multi-state life insurance, known as aggregate Markov models, with a specific focus on the valuation of liabilities. The use of matrix analytic methods lead to a unifying and transparent treatment of both the probabilistic as well as the computational facets of the models. In contrast, the companion paper~\cite{AhmadBladt} is devoted to the statistical aspects.

In summary, aggregate Markov modelling offers a versatile approach for modelling multi-state life insurance contracts, retaining the advantageous analytical features of Markov chain modelling while providing added flexibility, including the incorporation of duration effects. Our theoretical results suggest that aggregate Markov modelling may have a competitive advantage over semi-Markov modelling when the contractual payments exhibit a sufficient level of simplicity in their duration dependence. This assertion is further supported by a numerical example.

Future research endeavors could address several open questions. For instance, the resolution of Conjecture~\ref{conj:dense}, which proposes that any (smooth) semi-Markovian model can be approximated arbitrarily well by an aggregate Markov model with the reset property, remains unresolved. An affirmative solution, achieved through an efficient construction, would perhaps encourage the wider adoption of aggregate Markov models.

Throughout this study, our focus has been solely on smooth models, meaning models with absolutely continuous sojourn time distributions. However, to unify modelling in discrete and continuous time or to capture phenomena like stochastic retirement, the inclusion of point probability mass in the distribution of jumps becomes necessary. Such an extension -- employing cumulative transition rates as opposted to ordinary transition rates -- seems straightforward in principle, albeit requiring new proofs.

Finally, the computation of higher-order moments of present values holds relevance, particularly for risk management purposes, as it would enable for instance the calculation of quantiles. In the case of duration-independent payments, we can compute higher-order moments based on the microstate process $\vect{X}$, which is Markovian, following along the same lines as in~\cite{Bladt2020} and leveraging the favorable structure of~\eqref{eq:nice_moments}. However, the case of duration-dependent payments is more involved and would require additional work beyond the scope of this paper.}

%% file: log_liks.tex
% latex table generated in R 4.2.2 by xtable 1.8-4 package
% Sun Jun 18 18:13:10 2023
\begin{table}[h]
\centering
\begin{tabular}{c| rrrrrr}
 $d_2$ & 1 & 2 & 3 & 5 & 7 & 10 \\ \hline
 Log-likelihood & -26,315 & -24,324 & -24,241 & -24,153 & -24,132 & -24,115
  \end{tabular}
\caption{
{
Final log-likelihood values from the statistical analyses in~\cite{AhmadBladt}.}
} 
\label{tab:log-liks}
\end{table}

%% file: MarkovAdditiveAppendix.tex
\begin{appendices}

\section{Proofs}\label{ap:A}

This appendix contains the proofs of the results from Sections~\ref{sec:propZ}--\ref{sec:reserves_cf}. In the following, we denote by $\vect{e}_{\widetilde{j}}$ the $d_j$-dimensional vector with a one in entry $\widetilde{j}$ and otherwise zeroes. To prove Proposition~\ref{thm:MPP} and Theorem~\ref{thm:compensator}, we need the next lemma. 

\begin{lemma}\label{lemma:help}
For $t \geq t_n$ and $k \neq y_n$ it holds that
\begin{align*}
\mathbb{P}\big(t < T_{n+1} \leq t + h, \vect{X}(T_{n+1}) = \mat{\kay} \, \big| \, S_n = s_n \big)
=
\frac{\vect{\alpha}(s_n,t,k) \vect{e}_{\widetilde{k}}}{\vect{\alpha}(s_n)\vect{1}_{d_{y_n}}} h + o(h), \quad h \to 0,
\end{align*}
where the $d_{y_n}$-dimensional row vector $\vect{\alpha}(s_n)$ is given by
\begin{align*}
\vect{\alpha}(s_n) = \vect{\pi}_1(0)\displaystyle\prod_{\ell = 0}^{n-1}\,\Prodi_{t_{\ell}}^{t_{\ell+1}}\!\left(\mat{I}+\mat{M}_{y_\ell y_\ell}(x)\md x\right)\mat{M}_{y_\ell y_{\ell +1}}(t_{\ell+1}).
\end{align*}
\end{lemma}
\begin{proof}
We give a proof by induction. First, we verify the identity for $n=0$. Note that $\vect{\alpha}(0,1)\vect{1}_{d_1} = 1$ and that
\begin{align*}
&\mathbb{P}\big(t < T_1 \leq t + h, \vect{X}(T_1) = \mat{\kay}\big) \\
&=
\sum_{\widetilde{y}_0} \mathbb{P}\big(t < T_1 \leq t + h, \vect{X}(T_1) = \mat{\kay}, \vect{X}(t) = (1,\widetilde{y}_0)\big) \\
&=
\sum_{\widetilde{y}_0} \bigg( \Big[ o(h) + \mathbb{P}\big(T_1 \leq t + h, \vect{X}(t+h) = \mat{\kay} \, \big| \, t < T_1, \vect{X}(t) = (1,\widetilde{y}_0)\big) \Big] \\
&\qquad\mathbb{P}\big(t < T_1, \vect{X}(t) = (1,\widetilde{y}_0)\big) \bigg) \\
&=
o(h) + h \sum_{\widetilde{y}_0} \mu_{(1,\widetilde{y}_0)\mat{\kay}}(t) \vect{\pi}_1(0) \Prodi_0^t \!\left(\mat{I} + \mat{M}_{11}(x)\md x\right) \vect{e}_{\widetilde{y}_0} \\
&=
o(h) + h \, \vect{\pi}_1(0) \Prodi_0^t \!\left(\mat{I} + \mat{M}_{11}(x)\md x\right) \mat{M}_{1k}(t) \vect{e}_{\widetilde{k}} \\
&=
o(h) + h \, \vect{\alpha}(s_1) \vect{e}_{\widetilde{k}}.
\end{align*}
Collecting results confirms the identity for $n=0$. Now suppose the identity holds for $n\in\mathbb{N}_0$. We want to establish the identity also for $n+1$. By assumption,
\begin{align*}
&\mathbb{P}\big(t_{n+1} < T_{n+1} \leq t_{n+1} + \widetilde{h}, \vect{X}(T_{n+1}) = (y_{n+1},\check{y}_{n+1}) \, \big| \, S_n = s_n\big) \\
&=
\frac{\vect{\alpha}(s_n, t_{n+1}, y_{n+1})\vect{e}_{\check{y}_{n+1}}}{\vect{\alpha}(s_n)\vect{1}_{d_{y_n}}} \widetilde{h} + o(\widetilde{h}).
\end{align*}
In particular,
\begin{align}\label{eq:lemma_help}
\begin{split}
&\mathbb{P}\big(t_{n+1} < T_{n+1} \leq t_{n+1} + \widetilde{h}, Y_{n+1} = y_{n+1} \, \big| \, S_n = s_n\big) \\
&=
\frac{\vect{\alpha}(s_n, t_{n+1}, y_{n+1})\vect{1}_{d_{y_{n+1}}}}{\vect{\alpha}(s_n)\vect{1}_{d_{y_n}}} \widetilde{h} + o(\widetilde{h}).
\end{split}
\end{align}
Furthermore, 
\begin{align*}
&\mathbb{P}\big(t < T_{n+2}, \vect{X}(t) = \mat{y_{n+1}} \, \big| \, T_{n+1} = t_{n+1}, \vect{X}(t_{n+1}) = (y_{n+1},\check{y}_{n+1})\big) \\
&=
\vect{e}_{\check{y}_{n+1}}^{\prime} \Prodi_{t_{n+1}}^t \! \left( \mat{I} + \mat{M}_{y_{n+1}y_{n+1}}(x)\md x\right) \vect{e}_{\widetilde{y}_{n+1}},
\end{align*}
so that
\begin{align*}
&\mathbb{P}\big(t < T_{n+2}, \vect{X}(t) = \mat{y_{n+1}}, t_{n+1} < T_{n+1} \leq t_{n+1} + \widetilde{h}, Y_{n+1} = y_{n+1} \, \big| \, S_n = s_n\big)  \\
&=
\frac{\vect{\alpha}(s_n, t_{n+1}, y_{n+1})}{\vect{\alpha}(s_n)\vect{1}_{d_{y_n}}} \Prodi_{t_{n+1}}^t \! \left( \mat{I} + \mat{M}_{y_{n+1}y_{n+1}}(x)\md x\right) \vect{e}_{\widetilde{y}_{n+1}} \widetilde{h} + o(\widetilde{h}).
\end{align*}
Using this result, we find that
\begin{align*}
&\mathbb{P}\big( t < T_{n+2} \leq t + h, \vect{X}(T_{n+2}) = \mat{\kay}, t_{n+1} < T_{n+1} \leq t_{n+1} + \widetilde{h}, Y_{n+1} = y_{n+1} \, \big| \, S_n = s_n\big) \\
&=
\sum_{\widetilde{y}_{n+1}} \bigg( \Big[ o(h) + \mathbb{P}\big(T_{n+2} \leq t + h, \vect{X}(t+h) = \mat{\kay} \, \big| \, t < T_{n+2}, \vect{X}(t) = \mat{y_{n+1}}\big) \Big] \\
&\qquad\mathbb{P}\big(t < T_{n+2}, \vect{X}(t) = \mat{y_{n+1}}, t_{n+1} < T_{n+1} \leq t_{n+1} + \widetilde{h}, Y_{n+1} = y_{n+1} \, \big| \, S_n = s_n\big) \bigg) \\
&=
\sum_{\widetilde{y}_{n+1}} \bigg( \Big[o(h) + \mu_{\mat{y_{n+1}}\mat{\kay}}(t) h \Big] \\
&\qquad \Big[ \frac{\vect{\alpha}(s_n, t_{n+1}, y_{n+1})}{\vect{\alpha}(s_n)\vect{1}_{d_{y_n}}} \Prodi_{t_{n+1}}^t \! \left( \mat{I} + \mat{M}_{y_{n+1}y_{n+1}}(x)\md x\right) \vect{e}_{\widetilde{y}_{n+1}} \widetilde{h} + o(\widetilde{h})\Big]\bigg),
\end{align*}
Combining this result with~\eqref{eq:lemma_help} allows us to conclude that
\begin{align*}
&\mathbb{P}\big( t < T_{n+2} \leq t + h, \vect{X}(T_{n+2}) = \mat{\kay} \, \big| \, S_{n+1} = s_{n+1}\big) \\
&=
o(h)
+
\frac{h \sum_{\widetilde{y}_{n+1}} \frac{\vect{\alpha}(s_n,t_{n+1},y_{n+1})}{\vect{\alpha}(s_n)\vect{1}_{d_{y_n}}} \Prodi_{t_{n+1}}^t \! \left( \mat{I} + \mat{M}_{y_{n+1}y_{n+1}}(x)\md x\right) \vect{e}_{\widetilde{y}_{n+1}} \mu_{\mat{y_{n+1}}\mat{\kay}}(t)}
{\frac{\vect{\alpha}(s_n,t_{n+1},y_{n+1})\vect{1}_{d_{y_{n+1}}}}{\vect{\alpha}(s_n)\vect{1}_{d_{y_n}}}} \\
&=
o(h)
+
h \frac{\vect{\alpha}(s_{n+1})}{\vect{\alpha}(s_{n+1})\vect{1}_{d_{y_{n+1}}}} \Prodi_{t_{n+1}}^t \! \left( \mat{I} + \mat{M}_{y_{n+1}y_{n+1}}(x)\md x\right) \mat{M}_{y_{n+1}k}(t) \vect{e}_{\widetilde{k}} \\
&=
o(h) + \frac{\vect{\alpha}(s_{n+1},t,k) \vect{e}_{\widetilde{k}}}{\vect{\alpha}(s_{n+1})\vect{1}_{d_{y_{n+1}}}} h.
\end{align*}
This establishes the identity for $n+1$ and thus completes the proof.
\end{proof}

\begin{proof}[Proof of Proposition~\ref{thm:MPP}]
From Lemma~\ref{lemma:help} we immediately get that $\mathbb{P}(T_{n+1} \leq t, Y_{n+1} = k \, | \, S_n = s_n)$ is absolutely continuous with respect to the Lebesgue measure with density
\begin{align*}
f^{(n+1)}(t, k | s_n)
&=
\frac{\vect{\alpha}(s_n, t, k)\vect{1}_{d_k}}{\vect{\alpha}(s_n)\vect{1}_{d_{y_n}}} \\
&=
\frac{\vect{\alpha}(s_n)}{\vect{\alpha}(s_n)\vect{1}_{d_{y_n}}} \Prodi_{t_n}^t \left(\mat{I}+\mat{M}_{y_n y_n}(x)\md x\right)\mat{M}_{y_n k}(t) \vect{1}_{d_k}
\end{align*}
for $t > t_n$ and $k \neq y_n$. In particular, $1 - \bar{F}^{(n+1)}(t|s_n)$ is absolutely continuous with respect to the Lebesgue measure with density
\begin{align}\label{eq:density_help}
\begin{split}
f^{(n+1)}(t|s_n)
&=
\sum_{k \neq y_n} \frac{\vect{\alpha}(s_n)}{\vect{\alpha}(s_n)\vect{1}_{d_{y_n}}} \Prodi_{t_n}^t \left(\mat{I}+\mat{M}_{y_n y_n}(x)\md x\right)\mat{M}_{y_n k}(t) \vect{1}_{d_k} \\
&=
\frac{\vect{\alpha}(s_n)}{\vect{\alpha}(s_n)\vect{1}_{d_{y_n}}} \Prodi_{t_n}^t \left(\mat{I}+\mat{M}_{y_n y_n}(x)\md x\right)\mat{m}_{y_n}(t)
\end{split}
\end{align}
for $t > t_n$; confer also with~\eqref{eq:exit_rates}. Based on for instance the forward equations for product integrals, see~\cite[Proposition 5 and 6]{GillJohansen}, we may then argue that
\begin{align*}
\bar{F}^{(n+1)}(t|s_n) =   \frac{\vect{\alpha}(s_n)}{\vect{\alpha}(s_n)\vect{1}_{d_{y_n}}} \Prodi_{t_n}^t \left(\mat{I}+\mat{M}_{y_n y_n}(x)\md x\right)\vect{1}_{d_{y_n}}, \quad t \geq t_n,
\end{align*}
which proves the first assertion of the proposition. For the second part, we let $k\neq y_n$ and find that
\begin{align*}
G^{(n+1)}(k | s_n, t_{n+1})
&=
\frac{f^{(n+1)}(t_{n+1}, k | s_n)}{f^{(n+1)}(t_{n+1} | s_n)} \\
&=
\frac{
\frac{\vect{\alpha}(s_n)}{\vect{\alpha}(s_n)\vect{1}_{d_{y_n}}} \Prodi_{t_n}^{t_{n+1}} \left(\mat{I}+\mat{M}_{y_n y_n}(x)\md x\right)\mat{M}_{y_n k}(t_{n+1}) \vect{1}_{d_k}
}
{
\frac{\vect{\alpha}(s_n)}{\vect{\alpha}(s_n)\vect{1}_{d_{y_n}}} \Prodi_{t_n}^{t_{n+1}} \left(\mat{I}+\mat{M}_{y_n y_n}(x)\md x\right)\mat{m}_{y_n}(t_{n+1})
} \\
&=
\frac{
\vect{\alpha}(s_n) \Prodi_{t_n}^{t_{n+1}} \left(\mat{I}+\mat{M}_{y_n y_n}(x)\md x\right)\mat{M}_{y_n k}(t_{n+1}) \vect{1}_{d_k}
}
{
\vect{\alpha}(s_n) \Prodi_{t_n}^{t_{n+1}} \left(\mat{I}+\mat{M}_{y_n y_n}(x)\md x\right)\mat{m}_{y_n}(t_{n+1})
},
\end{align*} 
as desired.
\end{proof}
\begin{proof}[Proof of Theorem~\ref{thm:compensator}]
In the previous proof, we already noted that $1 - \bar{F}^{(n+1)}(t|s_n)$ is absolutely continuous with respect to the Lebesgue measure with density 
\begin{align*}
f^{(n+1)}(t|s_n) = \frac{\vect{\alpha}(s_n)}{\vect{\alpha}(s_n)\vect{1}_{d_{y_n}}} \Prodi_{t_n}^t \left(\mat{I}+\mat{M}_{y_n y_n}(x)\md x\right)\mat{m}_{y_n}(t)
\end{align*}
for $t > t_n$; see in particular~\eqref{eq:density_help}. Then $\mathrm{d}\Lambda_{jk}(t) = \lambda_{jk}(t) \, \mathrm{d}t$ with
\begin{align*}
\lambda_{jk} (t) = \sum_{n\in \N_0}1_{\left(T_{n} ,\, T_{n+1} \right]}(t)1_{\left(Y _{n} = j\right)} \frac{f^{(n+1)}\!\left(t|S_{n-1} , T_n , j  \right)}{\bar{F}^{(n+1)}\!\left(t|S_{n-1} , T_n , j  \right)}G^{(n+1)}\!\left(k|S _{n-1}, T _n, j, t\right)\!,
\end{align*}
confer with~\cite[Proposition 4.4.1(b)(ii)]{jacobsen2006}. The result now follows from inserting the expressions for~$\bar{F}^{(n+1)}$ and $G^{(n+1)}$ obtained in Proposition~\ref{thm:MPP} along with the above expression for $f^{(n+1)}$.  
\end{proof}
\begin{proof}[Proof of Theorem~\ref{thm:p}]
Due to the decomposition
\begin{align*}
\bar{p}_{\mat{\jay}}(t,s,z) = 
\sum_{n\in\mathbb{N}_0}
1_{\left[T_{n} ,\, T_{n+1} \right)}(t)\mathbb{P}\big(\vect{X}(s) = \mat{\jay}, U(s) > z \, | \, T_{n+1} > t, S_n\big),
\end{align*}
it suffices to show that
\begin{align*}
\mathbb{P}\big(\vect{X}(s) = \mat{\jay}, U(s) > z \, | \, T_{n+1} > t, S_n = s_n\big)
=
\widetilde{p}_{\mat{\jay}}(t,s,z; s_{n}),
\end{align*}
whenever they are well-defined. The case $t_n \geq s - z$ is trivial, so suppose in the following that $t_n < s - z$. We find that
\begin{align*}
&\mathbb{P}\big(\vect{X}(s) = \mat{\jay}, U(s) > z \, | \, T_{n+1} > t, S_n = s_n\big) \\
&=
\sum_{\widetilde{y}_n}
\frac
{
\mathbb{P}\big(\vect{X}(s) = \mat{\jay}, U(s) > z, T_{n+1} > t, \vect{X}(t) = \mat{y_n} \, | \, S_n = s_n\big)
}
{
\mathbb{P}\big(T_{n+1} > t \, | \, S_n = s_n\big)
} \\
&=
\sum_{\widetilde{y}_n} \bigg(
\frac
{
\mathbb{P}\big(T_{n+1} > t, \vect{X}(t) = \mat{y_n} \, | \, S_n = s_n\big)
}
{
\mathbb{P}\big(T_{n+1} > t \, | \, S_n = s_n\big)
} \\
&\quad\quad\quad\quad\mathbb{P}\big(\vect{X}(s) = \mat{\jay}, U(s) > z \, | \, T_{n+1} > t, \vect{X}(t) = \mat{y_n}, S_n = s_n\big) \bigg).
\end{align*}
According to the first statement of Proposition~\ref{thm:MPP},
\begin{align*}
\mathbb{P}\big(T_{n+1} > t \, | \, S_n = s_n\big)
=
\frac{\vect{\alpha}(s_n)}{\vect{\alpha}(s_n)\vect{1}_{d_{y_n}}} \Prodi_{t_n}^t \left(\mat{I}+\mat{M}_{y_n y_n}(x)\md x\right)\vect{1}_{d_{y_n}}.
\end{align*}
Also, using similar techniques as in the proof of Proposition~\ref{thm:MPP} and referring to Lemma~\ref{lemma:help}, one may show that
\begin{align*}
\mathbb{P}\big(T_{n+1} > t, \vect{X}(t) = \mat{y_n} \, | \, S_n = s_n\big)
=
\frac{\vect{\alpha}(s_n)}{\vect{\alpha}(s_n)\vect{1}_{d_{y_n}}} \Prodi_{t_n}^t \left(\mat{I}+\mat{M}_{y_n y_n}(x)\md x\right)\vect{e}_{\widetilde{y}_n}.
\end{align*}
If $s - z \leq t$, then
\begin{align*}
&\mathbb{P}\big(\vect{X}(s) = \mat{\jay}, U(s) > z \, | \, T_{n+1} > t, \vect{X}(t) = \mat{y_n}, S_n = s_n\big) \\
&=
1_{(j = y_n)}
\mathbb{P}\big(\vect{X}(s) = \mat{\jay}, T_{n+1} > s \, | \, T_{n+1} > t, \vect{X}(t) = \mat{y_n}, S_n = s_n\big) \\
&=
1_{(j = y_n)}
\vect{e}^{\prime}_{\widetilde{y}_n}\Prodi_t^s \left(\mat{I}+\mat{M}_{y_n y_n}(x)\md x\right)\vect{e}_{\widetilde{j}},
\end{align*}
so that
\begin{align*}
\mathbb{P}\big(\vect{X}(s) = \mat{\jay}, U(s) > z \, | \, T_{n+1} > t, S_n = s_n\big)
=
1_{(j = y_n)}
\frac
{
\vect{\alpha}(s_{n})\displaystyle \Prodi_{t_{n}}^s\!\left(\mat{I}+\mat{M}_{y_ny_n}(x)\mathrm{d} x\right)\vect{e}_{\widetilde{j}}
}
{
\vect{\alpha}(s_{n})\displaystyle \Prodi_{t_{n}}^t\!\left(\mat{I}+\mat{M}_{y_ny_n}(x)\mathrm{d} x\right)\vect{1}_{d_{y_n}}
},
\end{align*}
which exactly equals $\widetilde{p}_{\mat{\jay}}(t,s,z; s_{n})$, confer with~\eqref{eq:yn_is_j}. If instead $s - z > t$, then the Markov property of $\vect{X}$ yields
\begin{align*}
&\mathbb{P}\big(\vect{X}(s) = \mat{\jay}, U(s) > z \, | \, T_{n+1} > t, \vect{X}(t) = \mat{y_n}, S_n = s_n\big) \\
&=
\sum_{\check{j}}
\mathbb{P}\big(\vect{X}(s) = \mat{\jay}, U(s) > z, \vect{X}(s-z) = (j,\check{j}) \, | \, \vect{X}(t) = \mat{y_n}\big) \\
&=
\sum_{\check{j}}
\mathbb{P}\big(\vect{X}(s) = \mat{\jay}, U(s) > z \, | \, \vect{X}(s-z) = (j,\check{j})) \mathbb{P}(\vect{X}(s-z) = (j,\check{j}) \, | \, \vect{X}(t) = \mat{y_n}\big) \\
&=
\sum_{\check{j}} \vect{e}^{\prime}_{\check{j}} \Prodi_{s-z}^s\!\left(\mat{I}+\mat{M}_{jj}(x)\mathrm{d} x\right)\vect{e}_{\widetilde{j}} \vect{e}^{\prime}_{\mat{y_n}} \Prodi_t^{s-z}\!\left(\mat{I}+\mat{M}(x)\mathrm{d} x\right)\vect{e}_{\mat{\jay}} \\
&=
\vect{e}^{\prime}_{\mat{y_n}} \Prodi_t^{s-z}\!\left(\mat{I}+\mat{M}(x)\mathrm{d} x\right) \mat{E}_j  \Prodi_{s-z}^s\!\left(\mat{I}+\mat{M}_{jj}(x)\mathrm{d} x\right)\vect{e}_{\widetilde{j}}.
\end{align*}
Collecting results completes the proof.
\end{proof}

\end{appendices}